\documentclass[12pt]{amsart}
\usepackage{amsfonts}
\usepackage{latexsym,amssymb,amsmath,amsthm,amscd,graphicx}
\usepackage{color}
\usepackage{geometry}
\usepackage{accents}

\numberwithin{equation}{section}
\newtheorem{theorem}{Theorem}[section]
\newtheorem{lemma}[theorem]{Lemma}
\newtheorem{corollary}[theorem]{Corollary}
\newtheorem{proposition}[theorem]{Proposition}

\theoremstyle{definition}

\newtheorem{remark}[theorem]{Remark}
\newtheorem{definition}[theorem]{Definition}

\theoremstyle{remark}
\newtheorem{notation}{Notation}

\definecolor{blue}{rgb}{0,0,0.45}
\definecolor{red}{rgb}{0.7,0,0}
\input{tcilatex}

\begin{document}
\title[approximation by entire IFFD on the real line]{EXPONENTIAL
APPROXIMATION IN VARIABLE EXPONENT LEBESGUE SPACES ON THE REAL LINE}
\author{Ramazan Akg\"{u}n}
\maketitle

\begin{quotation}
\textbf{Abstract} Present work contains a method to obtain Jackson and
Stechkin type inequalities of approximation by integral functions of finite
degree (IFFD) in some variable exponent Lebesgue space of real functions
defined on $\boldsymbol{R}:=\left( -\infty ,+\infty \right) $. To do this we
employ a transference theorem which produce norm inequalities starting from
norm inequalities in $\mathcal{C}(\boldsymbol{R})$, the class of bounded
uniformly continuous functions defined on $\boldsymbol{R}$. Let $B\subseteq 
\boldsymbol{R}$ be a measurable set, $p\left( x\right) :B\rightarrow \lbrack
1,\infty )$ be a measurable function. For the class of functions $f$
belonging to variable exponent Lebesgue spaces $L_{p\left( x\right) }\left(
B\right) $ we consider difference operator $\left( I-T_{\delta }\right)
^{r}f\left( \cdot \right) $ under the condition that $p(x)$ satisfies the
Log H\"{o}lder continuity condition and $1\leq \mathop{\rm ess \; inf}%
\limits\nolimits_{x\in B}p(x)$, $\mathop{\rm ess \;
sup}\limits\nolimits_{x\in B}p(x)<\infty $ where $I$ is the identity
operator, $r\in \mathrm{N}:=\left\{ 1,2,3,\cdots \right\} $, $\delta \geq 0$
and%
\begin{equation*}
T_{\delta }f\left( x\right) =\frac{1}{\delta }\int\nolimits_{0}^{\delta
}f\left( x+t\right) dt\text{,\quad }x\in \boldsymbol{R}\text{,\quad }%
T_{0}\equiv I\text{,\quad \quad \quad \quad }(\ast )
\end{equation*}%
is the forward Steklov operator. It is proved that%
\begin{equation*}
\left\Vert \left( I-T_{\delta }\right) ^{r}f\right\Vert _{p\left( \cdot
\right) }\text{,\hspace{2.67in}}(\ast \ast )
\end{equation*}%
is a suitable measure of smoothness for functions in $L_{p\left( x\right)
}\left( B\right) $ where $\left\Vert \cdot \right\Vert _{p\left( \cdot
\right) }$ is Luxemburg norm in $L_{p\left( x\right) }\left( B\right) .$ We
obtain main properties of difference operator $\left\Vert \left( I-T_{\delta
}\right) ^{r}f\right\Vert _{p\left( \cdot \right) }$ in $L_{p\left( x\right)
}\left( B\right) .$ We give proof of direct and inverse theorems of
approximation by IFFD in $L_{p\left( x\right) }\left( \boldsymbol{R}\right)
. $


\end{quotation}

\section{Introduction}

Some inequalities of Approximation Theory in a Homogenous Banach Spaces
(HBS) can be obtained their uniform-norm counterparts. This information is
known for a long time, (see e.g., \cite{z76} for definition of HBS). This
elegant method was generalized to some variable exponent Lebesgue spaces
functions defined on $\boldsymbol{R}$. (see Theorem 21 of \cite{AkgArx17}).
Generally, these scale of function classes are non-translation invariant
with respect to the ordinary translation $x\rightarrow f\left( x+a\right) $.
Here we give several uniform-norm inequalities on $C(\boldsymbol{R})$ and
apply them to obtain several inequalities of approximation by IFFD in some
variable exponent Lebesgue spaces $L_{p(x)}(\boldsymbol{R})$. Under some
condition on $p(x)$ of $L_{p(x)}(\boldsymbol{R})$ we obtain main
inequalities of exponential approximation by IFFD such as
Jackson-Stechkin-Timan type estimates and equivalence of \textit{K}%
-functional with suitable modulus of smoothness $(\ast \ast )$ given in
abstract for functions of $L_{p(x)}(\boldsymbol{R})$. Note that many results
of approximation by IFFD can be obtained easily their uniform-norm
counterparts in $C(\boldsymbol{R})$.

Consider an entire function $f(z)$ and put $M(r)=\max_{|z|=r}|f(z)|$ for $%
z=x+iy$. We say that an entire function $f$ is of exponential type $\sigma $
if $\limsup_{r\rightarrow \infty }r^{-1}\ln M(r)\leq \sigma $,\quad $\sigma
<\infty .$

The approximation by entire function of finite degree in the real line was
originated in the beginning of twentieth century by Serge Bernstein \cite%
{Ber} and became a separate branch of analysis due to the efforts of many
mathematicians such as N. Wiener and R. Paley \cite{RN}, N.I. Ahiezer \cite%
{Ak}, S.M. Nikolskii \cite{Ni}, I.I. Ibragimov \cite{II3}, A. F. Timan \cite%
{AFT}, M. F. Timan \cite{MFT}, R. Taberski \cite{T81,T86}, F.G. Nasibov \cite%
{NF}, V. Yu. Popov \cite{Po}, A. A. Ligun \cite{LiDo}, and others.

Studying function spaces with variable exponent is now an extensively
developed field after their applications in elasticity theory \cite{vz86},
fluid mechanics \cite{krr-mr-96, mr00}, differential operators \cite%
{ld-mr02,mr00}, nonlinear Dirichlet boundary value problems \cite{zok-jr91},
nonstandard growth \cite{vz86}, and variational calculus. See the books \cite%
{UF13,DHHR11,Sh12} for more references. Nowadays many mathematician solved
many problems for the approximation of function in these type spaces defined
on $\left[ 0,2\pi \right] \subset \boldsymbol{R}$ (see e.g., \cite{AK2,AK3,
GI2, Ja, Ja1,AI}, \cite{acis,ascs,aosss,ahak,ayey}, \cite%
{ra11u,AK1,AkgArx17,bbsv06,B46},\cite{day,DQR03,xf-dz01,hh76,ja2,ja3,it16},%
\cite{iy,k,worl31,sgs94,iis79,ssv}). In this paper we propose generalized
our last results in \cite{AG} which we obtained a direct and inverse
theorems for approximation by entire functions of finite degree in variable
exponent Lebesgue spaces on the whole real axis $\boldsymbol{R}$ with%
\begin{equation}
\sup\limits_{0<h\leq \delta }\Vert (I-T_{h})f\Vert _{{p(\cdot )}}
\label{bnm}
\end{equation}%
as modulus of continuity $\Omega _{1}(f,\delta )_{{p(\cdot )}}$. Instead of (%
\ref{bnm}), here we will use%
\begin{equation}
\Vert (I-T_{\delta })^{r}f\Vert _{{p(\cdot )}}  \label{bnm1}
\end{equation}%
as modulus smoothness $\Omega _{r}(f,\delta )_{{p(\cdot )}}$ and we obtain
stronger Jackson inequality than obtained in \cite{AG}.

Let $B\subseteq \boldsymbol{R}$ be a measurable set and $p(x):B\rightarrow
\lbrack 1,\infty )$ be a measurable function. We define $\tilde{P}\left(
B\right) $ as the class of measurable functions $p(x)$ satisfying the
conditions 
\begin{equation}
1\leq p_{B}^{-}\text{:=}\mathop{\rm ess \; inf}\limits\nolimits_{x\in B}p(x)%
\text{,\quad }p_{B}^{+}\text{:=}\mathop{\rm ess \; sup}\limits\nolimits_{x%
\in B}p(x)<\infty .  \label{gag21}
\end{equation}%
We also set $p^{-}:=p_{\boldsymbol{R}}^{-}$ and $p^{+}:=p_{\boldsymbol{R}%
}^{+}$. We define the $L_{p(\cdot )}(B)$ as the set of all functions $%
f:B\rightarrow \boldsymbol{R}$ such that%
\begin{equation}
I_{p(\cdot ),B}\left( \frac{f}{\lambda }\right) :=\int_{B}\left\vert \frac{%
f(y)}{\lambda }\right\vert ^{p(y)}dy<\infty  \label{Ip}
\end{equation}%
for some $\lambda >0$. We set $I_{p(\cdot )}\left( f\right) :=I_{p(\cdot ),%
\boldsymbol{R}}\left( f\right) $. The set of of functions $L_{p(\cdot )}(B)$%
, with norm 
\begin{equation*}
\Vert f\Vert _{p(\cdot ),B}:=\inf \left\{ \eta >0:I_{p(\cdot ),B}\left( 
\frac{f}{\eta }\right) <1\right\}
\end{equation*}%
is Banach space. We set $L_{p(\cdot )}:=L_{p(\cdot )}(\boldsymbol{R})$.

For $i\in \mathrm{N}$, all constants $c_{i}\left( x,y,\cdots \right) $ will
be some positive number such that they depend on the parameters $x,y,\cdots $
given in the brackets. Also constants $c_{i}\left( x,y,\cdots \right) $ can
be change only when the parameters $x,y,\cdots $ change. Absolute constants $%
\mathbf{c}_{1},\mathbf{c}_{2},\ldots $ will not change in each occurance.

\begin{definition}
\label{Def1} For a measurable set $B\subseteq \boldsymbol{R}$, a measurable
function $p(\cdot ):B\rightarrow \boldsymbol{R}$ is said to locally log-H%
\"{o}lder continuous on $B$ if there is a positive constant $c_{1}\left(
p\right) $ such that%
\begin{equation}
|p(x)-p(y)|\log \left( e+1/|x-y|\right) \leq c_{1}\left( p\right) <\infty
\label{gag22}
\end{equation}%
for any $x,y\in B$. We say that $p$ satisfies log-H\"{o}lder decay condition
if there is a constant $c_{2}\left( p\right) >0$ and $p_{\infty }>1$ such
that%
\begin{equation}
|p(x)-p_{\infty }|\log \left( e+|x|\right) \leq c_{2}\left( p\right) <\infty
\label{gag23}
\end{equation}%
for any $x\in B.$

Define the class $P^{Log}\left( B\right) :=\left\{ p\in \tilde{P}\left(
B\right) :\frac{1}{p}\text{ is satisfy (\ref{gag22})-(\ref{gag23})}\right\} $%
. We set $c_{3}\left( p\right) :=\max \left\{ c_{1}\left( p\right)
,c_{2}\left( p\right) \right\} .$
\end{definition}

\begin{definition}
(\cite[p.96]{HH}) Let $\mathbb{N}:\mathbb{=}\left\{ 1,2,3,\cdots \right\} $
be natural numbers and $\mathbb{N}_{0}:=\mathbb{N\cup }\left\{ 0\right\} $.

(a) A family $Q$ of measurable sets $E\subset \boldsymbol{R}$ is called
locally $N$-finite ($N\in \mathbb{N}$) if 
\begin{equation*}
\sum_{E\in Q}\chi _{E}\left( x\right) \leq N
\end{equation*}%
almost everywhere in $\boldsymbol{R}$ where $\chi _{U}$ is the
characteristic function of the set $U$.

(b) A family $Q$ of open bounded sets $U\subset \boldsymbol{R}$ is locally $%
1 $-finite if and only if the sets $U\in Q$ are pairwise disjoint.

(c) Let $U\subset \boldsymbol{R}$ be \ a measurable set and%
\begin{equation*}
A_{U}f:=\frac{1}{\left\vert U\right\vert }\int\limits_{U}\left\vert f\left(
t\right) \right\vert dt.
\end{equation*}

(d) For a family $Q$ of open sets $U\subset \boldsymbol{R}$ we define
averaging operator by 
\begin{equation*}
T_{Q}:L_{loc}^{1}\rightarrow L^{0},
\end{equation*}%
\begin{equation*}
T_{Q}f\left( x\right) :=\sum_{U\in Q}\chi _{U}\left( x\right)
A_{U}f=\sum_{U\in Q}\frac{\chi _{U}\left( x\right) }{\left\vert U\right\vert 
}\int\limits_{U}\left\vert f\left( y\right) \right\vert dy,\quad x\in 
\boldsymbol{R},
\end{equation*}%
where $L^{0}$ is the set of measurable functions on $\boldsymbol{R}$.
\end{definition}

For a measurable set $A\subset \boldsymbol{R}$, symbol $\left\vert
A\right\vert $ will represent the Lebesgue measure of $A$.

We consider Transference Result.

\begin{definition}
For $0<\delta <\infty $, $\tau \in \boldsymbol{R}$, we define family of
Steklov operators%
\begin{equation}
\text{\textsf{$\QTR{sc}{S}$}}_{\delta }f(x)\text{:=}\frac{1}{\delta }%
\int\nolimits_{x-\delta /2}^{x+\delta /2}f\left( t\right) dt\text{=}\frac{1}{%
\delta }\int\nolimits_{-\delta /2}^{\delta /2}f\left( x+t\right) dt\text{%
,\quad }x\in \boldsymbol{R},  \label{steklR}
\end{equation}%
where $f$ is a locally integrable function, defined on $\boldsymbol{R}$.
\end{definition}

The following result was obtained by Drihem for every cubes or balls in $%
\boldsymbol{R}^{d}$. We write below its restricted version with constants.
The proof of this is the same with Theorem 2 of \cite{drh20}.

\begin{proposition}
\label{dri}(\cite{drh20}) Suppose that $p\in P^{Log}\left( \boldsymbol{R}%
\right) $ and $Q$ is a bounded interval of $\boldsymbol{R}$ having Lebesgue
measure greater than $1$. For every $m>0$ there is $c_{4}\left(
m,c_{3}\left( p\right) \right) :=\exp \left( -4mc_{3}\left( p\right) \right)
\in \left( 0,1\right) $ such that%
\begin{equation*}
\left( \frac{c_{4}\left( m,c_{3}\left( p\right) \right) }{\left\vert
Q\right\vert }\int\limits_{Q}\left\vert f\left( y+\tau \right) \right\vert
dy\right) ^{p\left( x\right) }\leq \frac{3^{p^{+}}}{\left\vert Q\right\vert }%
\int\limits_{Q}\left\vert f\left( y+\tau \right) \right\vert ^{p\left(
y+\tau \right) }dy+\frac{3^{p^{+}-1}}{\left( e+\left\vert x\right\vert
\right) ^{m}}+
\end{equation*}%
\begin{equation*}
+3^{p^{+}-1}\int\limits_{Q}\frac{dy}{\left( e+\left\vert y+\tau \right\vert
\right) ^{m}}
\end{equation*}%
holds for all $x\in Q$, $\tau \in \boldsymbol{R}$ and all $f\in L_{p\left(
\cdot \right) }+L_{\infty }\left( \boldsymbol{R}\right) $ with $\left\Vert
f\right\Vert _{p\left( \cdot \right) }+\left\Vert f\right\Vert _{\infty
}\leq 1.$
\end{proposition}

\begin{theorem}
\label{stekRR} Suppose that $p\in P^{Log}\left( \boldsymbol{R}\right) $.
Then, the family of operators $\{\mathcal{U}_{\tau }f\}_{\tau \in 
\boldsymbol{R}}$, defined by 
\begin{equation*}
\mathcal{U}_{\tau }f(x):=\mathsf{S}_{1}f\left( x+\tau \right)
=\int\nolimits_{-1/2}^{+1/2}f\left( x+\tau +t\right) dt,\quad x\in 
\boldsymbol{R},\quad \tau \in \boldsymbol{R}
\end{equation*}%
is uniformly bounded (in $\tau $) in $L_{p\left( \cdot \right) }$, namely, 
\begin{equation*}
\left\Vert \mathcal{U}_{\tau }f\right\Vert _{p\left( \cdot \right) }\leq
c_{5}\left( p^{+},c_{3}\left( p\right) \right) \left\Vert f\right\Vert
_{p\left( \cdot \right) },
\end{equation*}%
holds with $c_{5}\left( p^{+},c_{3}\left( p\right) \right)
:=2^{p^{+}+1}3^{p^{+}}\left( 1+2\cdot 3^{p^{+}}\left[ \sum\nolimits_{k=2}^{%
\infty }2^{-k}+2\right] \right) \exp \left( 8c_{3}\left( p\right) \right) $.
\end{theorem}

\begin{proof}[\textbf{Proof of Theorem \protect\ref{stekRR}}]
Let us consider $f\in L_{p\left( \cdot \right) }$ with $\left\Vert
f\right\Vert _{p\left( \cdot \right) }\leq 1/2$. Suppose that $Q:=\left\{
U\subset \boldsymbol{R}:U\text{ open interval and }\left\vert U\right\vert
=1\right\} $ be a locally $1$-finite family of partition of $\boldsymbol{R}$%
. Choose $m=2>1$ ( constant $c_{6}\left( p^{+}\right) $ below becomes a
finite number)%
\begin{equation*}
c_{6}\left( p^{+}\right) =2^{p^{+}}3^{p^{+}}\left( 1+2\cdot 3^{p^{+}}\left[
\sum\nolimits_{k=2}^{\infty }2^{-k}+2\right] \right) <\infty .
\end{equation*}%
We can select $c_{4}\left( 2,c_{3}\left( p\right) \right) =\exp \left(
-8c_{3}\left( p\right) \right) \in \left( 0,1\right) $ as in Proposition \ref%
{dri}. Then using Corollary 2.2.2 of \cite[p.20]{HH} we obtain%
\begin{equation*}
\rho _{p\left( \cdot \right) }\left( \frac{c_{4}\left( 2,c_{3}\left(
p\right) \right) }{c_{6}\left( p^{+}\right) }\mathcal{U}_{\tau }f\right) 
\text{=}\frac{1}{c_{6}\left( p^{+}\right) }\int\limits_{\boldsymbol{R}%
}\left\vert c_{4}\left( 2,c_{3}\left( p\right) \right)
\int\nolimits_{-1/2}^{+1/2}f\left( x\text{+}\tau \text{+}t\right)
dt\right\vert ^{p\left( x\right) }dx
\end{equation*}%
\begin{equation*}
\leq \frac{1}{c_{6}\left( p^{+}\right) }\sum_{U\in
Q}\int\limits_{U}\left\vert c_{4}\left( 2,c_{3}\left( p\right) \right)
\int\nolimits_{-1/2}^{+1/2}f\left( x+\tau +t\right) dt\right\vert ^{p\left(
x\right) }dx
\end{equation*}%
\begin{equation*}
\leq \frac{2^{p^{+}}}{c_{6}\left( p^{+}\right) }\sum_{U\in
Q}\int\limits_{U}\left\vert \frac{c_{4}\left( 2,c_{3}\left( p\right) \right) 
}{\left\vert 2U\right\vert }\int\limits_{2U}\chi _{2U}\left( y\right)
f\left( y+\tau \right) dy\right\vert ^{p\left( x\right) }dx
\end{equation*}%
\begin{equation*}
\leq \frac{2^{p^{+}}}{c_{6}\left( p^{+}\right) }\sum_{U\in Q}\int\limits_{U}%
\left[ \frac{3^{p^{+}}\chi _{2U}\left( y\right) }{\left\vert 2U\right\vert }%
\int\limits_{2U}\left\vert f\left( y+\tau \right) \right\vert ^{p\left(
y+\tau \right) }dy+\right.
\end{equation*}%
\begin{equation*}
+\left. \frac{3^{p^{+}-1}}{\left( e+\left\vert x\right\vert \right) ^{2}}+%
\frac{\chi _{2U}\left( y\right) }{\left\vert 2U\right\vert }\int\limits_{2U}%
\frac{3^{p^{+}-1}dy}{\left( e+\left\vert y+\tau \right\vert \right) ^{2}}%
\right] dx
\end{equation*}%
\begin{equation*}
\leq \frac{2^{p^{+}-1}3^{p^{+}}}{c_{6}\left( p^{+}\right) }\sum_{U\in
Q}\int\limits_{U}\left[ \chi _{2U}\left( y\right) \int\limits_{2U+\tau
}\left\vert f\left( s\right) \right\vert ^{p\left( s\right) }ds\text{+}\frac{%
3^{p^{+}-1}2}{\left( e+\left\vert x\right\vert \right) ^{2}}\text{+}%
\int\limits_{2U+\tau }\frac{3^{p^{+}-1}ds}{\left( e+\left\vert s\right\vert
\right) ^{2}}\right] dx
\end{equation*}%
\begin{equation*}
\leq \frac{2^{p^{+}-1}3^{p^{+}}}{c_{6}\left( p^{+}\right) }\left( \sum_{U\in
Q}\chi _{2U}\right) \left( 1+3^{p^{+}}\int\limits_{\boldsymbol{R}}\frac{ds}{%
\left( e+\left\vert s\right\vert \right) ^{2}}\right)
\end{equation*}%
\begin{eqnarray*}
&=&\frac{2^{p^{+}}3^{p^{+}}}{c_{6}\left( p^{+}\right) }\left(
1+3^{p^{+}}\int\limits_{\boldsymbol{R}}\frac{ds}{\left( e+\left\vert
s\right\vert \right) ^{2}}\right) \\
&\leq &\frac{2^{p^{+}}3^{p^{+}}}{c_{6}\left( p^{+}\right) }\left( 1+2\cdot
3^{p^{+}}\left[ \sum\nolimits_{k=2}^{\infty }\frac{1}{2^{k}}+2\right]
\right) =1
\end{eqnarray*}%
and hence%
\begin{equation*}
\left\Vert \mathcal{U}_{\tau }f\right\Vert _{p\left( \cdot \right) }\leq
2^{-1}c_{5}\left( p^{+},c_{3}\left( p\right) \right) .
\end{equation*}%
General case $f\in L_{p\left( \cdot \right) }$ can be obtained easily by
rescaling:%
\begin{equation*}
\left\Vert \mathcal{U}_{\tau }f\right\Vert _{p\left( \cdot \right) }\leq
c_{5}\left( p^{+},c_{3}\left( p\right) \right) \left\Vert f\right\Vert
_{p(\cdot )}.
\end{equation*}
\end{proof}

\begin{theorem}
\label{Aver} (\cite[Theorem 4.4.8]{DHHR11}) Suppose that $p\in P^{Log}\left( 
\boldsymbol{R}\right) $ and $f\in L_{p\left( \cdot \right) }$. If $Q$ is
locally $1$-finite family of open bounded subintervals of $\boldsymbol{R}$
having Lebesgue measure $1$, then, the averaging operator $T_{Q}$ is
uniformly bounded in $L_{p\left( \cdot \right) }$, namely,%
\begin{equation*}
\left\Vert T_{Q}f\right\Vert _{p\left( \cdot \right) }\leq c_{7}\left(
c_{\log }\left( p\right) \right) \left\Vert f\right\Vert _{p\left( \cdot
\right) }
\end{equation*}%
holds with $c_{7}\left( c_{3}\left( p\right) \right) :=2\exp \left(
8c_{3}\left( p\right) \right) .$
\end{theorem}

Let $C(\boldsymbol{R})$ be the class of continuous functions defined on $%
\boldsymbol{R}$. For $r\in \mathrm{N}$, we define $C^{r}$ consisting of
every member $f\in C(\boldsymbol{R})$ such that the derivative $f^{\left(
k\right) }$ exists and is continuous\ on $\boldsymbol{R}$ for $k=1,...,r$.
We set $C^{\infty }:=\left\{ f\in C^{r}\text{ for any }r\in \mathrm{N}%
\right\} $. We denote by $C_{c}\left( \boldsymbol{R}\right) $, the
collection of real valued continuous functions on $\boldsymbol{R}$ and
support of $f$ is compact set in $\boldsymbol{R}.$ We define $%
C_{c}^{r}:=C^{r}\cap C_{c}\left( \boldsymbol{R}\right) $ for $r\in \mathrm{N}
$ and $C_{c}^{\infty }:=C^{\infty }\cap C_{c}\left( \boldsymbol{R}\right) $.
Let $L_{p}\left( \boldsymbol{R}\right) $, $1\leq p\leq \infty $ be the
classical Lebesgue space of functions on $\boldsymbol{R}$.

\begin{theorem}
\label{du} \cite[Corollary 4.6.6]{DHHR11} Let $p\in P^{Log}\left( 
\boldsymbol{R}\right) $ and $f\in L_{p\left( \cdot \right) }$. Then%
\begin{equation}
\frac{\left\Vert f\right\Vert _{p\left( \cdot \right) }}{12c_{7}\left(
c_{3}\left( p\right) \right) }\leq \sup_{g\in L_{p^{\prime }\left( \cdot
\right) }\cap C_{c}^{\infty }:\left\Vert g\right\Vert _{p^{\prime }\left(
\cdot \right) }\leq 1}\int\nolimits_{\boldsymbol{R}}\left\vert f\left(
x\right) g\left( x\right) \right\vert dx\leq 2\left\Vert f\right\Vert
_{p\left( \cdot \right) }  \label{Cdual}
\end{equation}
\end{theorem}

\begin{definition}
Let $p\in P^{Log}\left( \boldsymbol{R}\right) $. For an $f\in L_{p\left(
\cdot \right) }$, we define%
\begin{equation}
F_{f}\left( u\right) \text{:}\text{=}\int\nolimits_{\boldsymbol{R}}\left( 
\mathsf{S}_{1\text{ }}f\right) \left( x+u\right) \left\vert G\left( x\right)
\right\vert dx,\quad u\in \boldsymbol{R}\text{,}  \label{efef}
\end{equation}%
where $G\in L_{p^{\prime }\left( \cdot \right) }\cap C_{c}^{\infty }$ and $%
\left\Vert G\right\Vert _{p^{\prime }\left( \cdot \right) }\leq 1.$
\end{definition}

Let $W_{p\left( \cdot \right) }^{r}$, $r\in \mathrm{N}$, be the class of
functions $f\in L_{p\left( \cdot \right) }$ such that derivatives $f^{\left(
k\right) }$ exist for $k=1,...,r-1$, $f^{\left( r-1\right) }$ absolutely
continuous and $f^{\left( r\right) }\in L_{p\left( \cdot \right) }$.

Some properties of the function $F_{f}\left( \cdot \right) $ is given in the
following theorem.

\begin{theorem}
\label{Fu} Let $p\in P^{Log}\left( \boldsymbol{R}\right) $, $0<\delta
<\infty $, and $f\in L_{p\left( \cdot \right) }$. Then,

(a) the function $F_{f}\left( \cdot \right) $ defined in (\ref{efef}) is a
bounded, uniformly continuous on $\boldsymbol{R}$,

(b) $\left( \mathsf{S}_{\delta }f\right) ^{\prime }=\mathsf{S}_{\delta
}\left( f^{\text{ }\prime }\right) $ on $\boldsymbol{R}$ for $f\in
W_{p\left( \cdot \right) }^{1}$.
\end{theorem}

Main theorem of this section is as follows.

\begin{theorem}
\label{tra} Let $p\in P^{Log}\left( \boldsymbol{R}\right) $. If $f,g\in
L_{p\left( \cdot \right) }$ and%
\begin{equation*}
\left\Vert F_{f}\right\Vert _{C\left( \boldsymbol{R}\right) }\leq \mathbf{c}%
_{1}\left\Vert F_{g}\right\Vert _{C\left( \boldsymbol{R}\right) },
\end{equation*}%
holds with an absolute constant $\mathbf{c}_{1}>0$, then, norm inequality%
\begin{equation}
\left\Vert f\right\Vert _{p\left( \cdot \right) }\leq c_{8}\left( \mathbf{c}%
_{1},p^{+},c_{3}\left( p\right) \right) \left\Vert g\right\Vert _{p\left(
\cdot \right) }  \label{trak}
\end{equation}%
also holds with $c_{8}\left( \mathbf{c}_{1},p^{+},c_{3}\left( p\right)
\right) :=48c_{7}\left( c_{3}\left( p\right) \right) \mathbf{c}%
_{1}c_{5}\left( p^{+},c_{3}\left( p\right) \right) $.
\end{theorem}

For the troof of these results we will need the following Propositions.

\begin{proposition}
\label{pr2}(a) $C_{c}\left( \boldsymbol{R}\right) $ and $C_{c}^{\infty }$
are dense subsets of $L_{p}\left( \boldsymbol{R}\right) $, $1\leq p<\infty .$%
(Theorems 17.10 and 23.59 of \cite[p. 415 and p. 575]{yeh}).

(b) $C_{c}\left( \boldsymbol{R}\right) $ contained $L_{\infty }\left( 
\boldsymbol{R}\right) $ but not dense (Remark 17.11 of \cite[p.416]{yeh}) in 
$L_{\infty }\left( \boldsymbol{R}\right) .$

(c) if $r\in \mathrm{N}$ and $f\in C_{c}^{r}$, then, $\mathsf{S}_{\delta
}\left( f\right) \in C_{c}^{r}$.
\end{proposition}

\begin{proposition}
(\cite[Theorem 2.26]{DHHR11}) Let $B\subseteq \boldsymbol{R}$ be a
measurable set. If $1\leq p(x)<p_{B}^{+}<\infty $, $p^{\prime
}(x)=p(x)/(p(x)-1)$, $f\in L_{p(\cdot )}(B)$, and $g\in L_{p^{\prime }(\cdot
)}(B)$, then, H\"{o}lder's inequality%
\begin{equation}
\int\nolimits_{B}f(x)g(x)dx\leq 2\left\Vert f\right\Vert _{p(\cdot
),B}\left\Vert g\right\Vert _{p^{\prime }(\cdot ),B}  \label{holder}
\end{equation}%
holds$.$
\end{proposition}

\subsubsection{Proofs of Section 1}

\begin{proof}[Proof of Proposition \protect\ref{pr2}]
(a) and (b) are known. (c) is follows from definitions.
\end{proof}

\begin{proof}[Proof of Theorem \protect\ref{Fu}]
(a) Since $C_{c}\left( \boldsymbol{R}\right) $ is a dense subset (\cite[%
Theorem 4.1 (I)]{kosa03gmj}) of $L_{p\left( \cdot \right) }$, we consider
functions $H\in C_{c}\left( \boldsymbol{R}\right) $ and prove that $%
F_{H}\left( \cdot \right) =\int\nolimits_{\boldsymbol{R}}(S_{1}H)\left(
x+u_{1}\right) \left\vert G\left( x\right) \right\vert dx$ is bounded and
uniformly continuous on $\boldsymbol{R}$, where $G\in L_{p^{\prime }\left(
\cdot \right) }\cap C_{c}^{\infty }$ and $\left\Vert G\right\Vert
_{p^{\prime }\left( \cdot \right) }\leq 1$. Boundedness of $F_{H}\left(
\cdot \right) $ is easy consequence of the H\"{o}lder's inequality (\ref%
{holder}) and Theorem \ref{stekRR}. On the other hand, note that $H$ is
uniformly continuous on $\boldsymbol{R}$, see e.g. Lemma 23.42 of \cite[%
pp.557-558]{yeh}. Take $\varepsilon >0$ and $u_{1},u_{2},x\in \boldsymbol{R}$%
. Then, there exists a $\delta :=\delta \left( \varepsilon \right) >0$ such
that%
\begin{equation*}
\left\vert H\left( x+u_{1}\right) -H\left( x+u_{2}\right) \right\vert \leq 
\frac{\varepsilon }{2\left( 1+\left\vert \text{supp}\left( G\right)
\right\vert \right) }
\end{equation*}%
for $\left\vert u_{1}-u_{2}\right\vert <\delta $. Then, for $\left\vert
u_{1}-u_{2}\right\vert <\delta $, $u_{1},u_{2}\in \boldsymbol{R}$ we have%
\begin{equation*}
\left\vert F_{H}\left( u_{1}\right) -F_{H}\left( u_{2}\right) \right\vert
=\left\vert \int\nolimits_{\boldsymbol{R}}S_{1}\left( H\left( x+u_{1}\right)
-H\left( x+u_{2}\right) \right) \left\vert G\left( x\right) \right\vert
dx\right\vert
\end{equation*}%
\begin{equation*}
\leq \frac{1}{2\left( 1+\left\vert \text{supp}\left( G\right) \right\vert
\right) }\int\nolimits_{\boldsymbol{R}}\left\vert S_{1}\left( \varepsilon
\right) \right\vert \left\vert G\left( x\right) \right\vert dx=\frac{%
\varepsilon }{2\left( 1+\left\vert \text{supp}\left( G\right) \right\vert
\right) }\int\nolimits_{\boldsymbol{R}}\left\vert G\left( x\right)
\right\vert dx
\end{equation*}%
\begin{equation*}
\leq \frac{\varepsilon }{\left( 1+\left\vert \text{supp}\left( G\right)
\right\vert \right) }\left( 1+\left\vert \text{supp}\left( G\right)
\right\vert \right) \left\Vert G\right\Vert _{p^{\prime }\left( \cdot
\right) }\leq \varepsilon .
\end{equation*}%
Now, the conclusion of Theorem \ref{Fu} follows for the class $C_{c}\left( 
\boldsymbol{R}\right) $. For the general case $f\in L_{p\left( \cdot \right)
}$ there exists an $H\in C_{c}\left( \boldsymbol{R}\right) $ so that 
\begin{equation*}
\left\Vert f-H\right\Vert _{p\left( \cdot \right) }<\xi /8c
\end{equation*}%
for any $\xi >0.$ Then, for this $\xi $,%
\begin{equation*}
\left\vert F_{f}\left( u_{1}\right) -F_{f}\left( u_{2}\right) \right\vert
\leq \left\vert \int\nolimits_{\boldsymbol{R}}\mathsf{S}_{1}(f-H)\left(
x+u_{1}\right) \left\vert G\left( x\right) \right\vert dx\right\vert +
\end{equation*}%
\begin{equation*}
+\left\vert \int\nolimits_{\boldsymbol{R}}\mathsf{S}_{1}\left( H\left(
x+u_{1}\right) -H\left( x+u_{2}\right) \right) \left\vert G\left( x\right)
\right\vert dx\right\vert
\end{equation*}%
\begin{equation*}
+\left\vert \int\nolimits_{\boldsymbol{R}}\mathsf{S}_{1}(H-f)\left(
x+u_{2}\right) \left\vert G\left( x\right) \right\vert dx\right\vert
\end{equation*}%
\begin{equation*}
\leq 2\left\Vert \mathsf{S}_{1}\left( f-H\right) \left( \cdot +u_{1}\right)
\right\Vert _{p\left( \cdot \right) }+\left\vert \int\nolimits_{\boldsymbol{R%
}}\mathsf{S}_{1}\left( H\left( x+u_{1}\right) -H\left( x+u_{2}\right)
\right) \left\vert G\left( x\right) \right\vert dx\right\vert
\end{equation*}%
\begin{equation*}
+2\left\Vert \mathsf{S}_{1}\left( f-H\right) \left( \cdot +u_{2}\right)
\right\Vert _{p\left( \cdot \right) }
\end{equation*}%
\begin{equation*}
\leq 4c\left\Vert f-H\right\Vert _{p\left( \cdot \right) }+\xi /2\leq \xi
/2+\xi /2=\xi
\end{equation*}%
As a result $F_{f}$ is bounded, uniformly continuous function defined on $%
\boldsymbol{R}.$

(b) can be obtained easily from definition.
\end{proof}

\begin{proof}[Proof of Theorem \protect\ref{tra}]
Let $f\in L_{p\left( \cdot \right) }$ be nonnegative. If $\left\Vert
f\right\Vert _{p\left( \cdot \right) }=0$, then, the result (\ref{trak}) is
obvious. So we assume that $\infty >\left\Vert f\right\Vert _{p\left( \cdot
\right) }>0$. In this case%
\begin{align*}
\left\Vert F_{f}\right\Vert _{C\left( \boldsymbol{R}\right) }& \leq \mathbf{c%
}_{1}\left\Vert F_{g}\right\Vert _{C\left( \boldsymbol{R}\right) }=\mathbf{c}%
_{1}\left\Vert \int\nolimits_{\boldsymbol{R}}\mathrm{S}_{1}\left( g\right)
\left( u+x\right) \left\vert G\left( x\right) \right\vert dx\right\Vert
_{C\left( \boldsymbol{R}\right) } \\
& =\mathbf{c}_{1}\max_{u\in \boldsymbol{R}}\left\vert \int\nolimits_{%
\boldsymbol{R}}\mathrm{S}_{1}\left( g\right) \left( u+x\right) \left\vert
G\left( x\right) \right\vert dx\right\vert \\
& \leq 2\mathbf{c}_{1}\max_{u\in \boldsymbol{R}}\left\Vert \mathrm{S}%
_{1}\left( g\right) \left( u+\cdot \right) \right\Vert _{p\left( \cdot
\right) }\leq 2c_{5}\left( p^{+},c_{3}\left( p\right) \right) \mathbf{c}%
_{1}\left\Vert g\right\Vert _{p\left( \cdot \right) }
\end{align*}%
where we used hypotesis, H\"{o}lder's inequality and Theorem \ref{stekRR}
respectively. On the other hand, for any $\varepsilon \in \left( 0,\frac{%
\left\Vert f\right\Vert _{p\left( \cdot \right) }}{12c_{7}\left( c_{3}\left(
p\right) \right) }\right) $ and appropriately chosen $\tilde{G}_{\varepsilon
}\in L_{p^{\prime }\left( \cdot \right) }$ with $\left\Vert \tilde{G}%
_{\varepsilon }\right\Vert _{X^{\prime }}\leq 1$ (see e.g. Theorem \ref{du})%
\begin{equation*}
\int\nolimits_{\boldsymbol{R}}\left\vert g\left( x\right) \right\vert
\left\vert \tilde{G}_{\varepsilon }\left( x\right) \right\vert dx\geq \frac{1%
}{12c_{7}\left( c_{3}\left( p\right) \right) }\left\Vert g\right\Vert
_{p\left( \cdot \right) }-\varepsilon \text{,}
\end{equation*}%
one can find%
\begin{eqnarray*}
\left\Vert F_{f}\right\Vert _{C\left( \boldsymbol{R}\right) } &\geq
&\left\vert F_{f}\left( 0\right) \right\vert \geq \int\nolimits_{\boldsymbol{%
R}}\mathsf{S}_{1}\left( f\right) \left( x\right) \left\vert G\left( x\right)
\right\vert dx \\
&=&\mathsf{S}_{1}\left( \int\nolimits_{\boldsymbol{R}}f\left( x\right)
\left\vert G\left( x\right) \right\vert dx\right) \geq \mathsf{S}_{1}\left( 
\frac{1}{12c_{7}\left( c_{3}\left( p\right) \right) }\left\Vert f\right\Vert
_{p\left( \cdot \right) }-\varepsilon \right) \\
&=&\frac{1}{12c_{7}\left( c_{3}\left( p\right) \right) }\left\Vert
f\right\Vert _{p\left( \cdot \right) }-\varepsilon .
\end{eqnarray*}%
In the last inequality we take as $\varepsilon \rightarrow 0+$ and obtain%
\begin{equation*}
\left\Vert F_{f}\right\Vert _{C\left( \boldsymbol{R}\right) }\geq \frac{1}{%
12c_{7}\left( c_{3}\left( p\right) \right) }\left\Vert f\right\Vert
_{p\left( \cdot \right) }\text{.}
\end{equation*}%
Then for $f\in L_{p\left( \cdot \right) }$ we get%
\begin{eqnarray*}
\left\Vert f\right\Vert _{p\left( \cdot \right) } &\leq &24c_{7}\left(
c_{3}\left( p\right) \right) \left\Vert F_{f}\right\Vert _{C\left( 
\boldsymbol{R}\right) }\leq 24c_{7}\left( c_{3}\left( p\right) \right) 
\mathbf{c}_{1}\left\Vert F_{g}\right\Vert _{C\left( \boldsymbol{R}\right) }
\\
&\leq &48c_{7}\left( c_{3}\left( p\right) \right) \mathbf{c}_{1}c_{5}\left(
p^{+},c_{3}\left( p\right) \right) \left\Vert g\right\Vert _{p\left( \cdot
\right) }.
\end{eqnarray*}
\end{proof}

\begin{remark}
Theorem \ref{tra} is a powerful tool to obtain norm inequalities in $%
L_{p\left( \cdot \right) }$ (and other non-translation invariant Banach
spaces of functions) for $p\in P^{Log}\left( \boldsymbol{R}\right) $. In
this work we will use it frequently. See for example the following result.
\end{remark}

As a corollaries of Theorem \ref{tra} we get the following two results:

\begin{theorem}
\label{stekRRR} Suppose that $p\in P^{Log}\left( \boldsymbol{R}\right) $, $%
0<\delta <\infty $ and $\tau \in \boldsymbol{R}$. Then, the family of
operators $\{\mathcal{S}_{\delta ,\tau }f\}$, defined by 
\begin{equation*}
\mathcal{S}_{\delta ,\tau }f(x):=\mathsf{S}_{\delta }f\left( \cdot +\tau
\right) =\frac{1}{\delta }\int\nolimits_{x+\tau -\delta /2}^{x+\tau +\delta
/2}f\left( s\right) ds,\quad x\in \boldsymbol{R},
\end{equation*}%
is uniformly bounded (in $\delta $ and $\tau $) in $L_{p\left( \cdot \right)
}$, namely, 
\begin{equation*}
\left\Vert \mathcal{S}_{\delta ,\tau }f\right\Vert _{p\left( \cdot \right)
}\leq 48c_{7}\left( c_{3}\left( p\right) \right) c_{5}\left(
p^{+},c_{3}\left( p\right) \right) \left\Vert f\right\Vert _{p\left( \cdot
\right) }
\end{equation*}%
holds.
\end{theorem}

\begin{corollary}
\label{coroL} Let $p\in P^{Log}\left( \boldsymbol{R}\right) $, $0<\delta
<\infty $, and $f\in L_{p\left( \cdot \right) }.$ If $\tau =\delta /2$ then,%
\begin{equation*}
\mathcal{S}_{\delta ,\delta /2}f\left( x\right) =\frac{1}{\delta }%
\int\nolimits_{0}^{\delta }f\left( x+t\right) dt=T_{\delta }f\left( x\right)
,
\end{equation*}%
\begin{equation}
\left\Vert T_{\delta }f\right\Vert _{p\left( \cdot \right) }\leq
48c_{7}\left( c_{3}\left( p\right) \right) c_{5}\left( p^{+},c_{3}\left(
p\right) \right) \left\Vert f\right\Vert _{p\left( \cdot \right) }\text{ and}
\label{hurr}
\end{equation}%
\begin{equation*}
\left\Vert (I-T_{\delta })^{r}f\right\Vert _{p\left( \cdot \right) }\leq
\left( 1+48c_{7}\left( c_{3}\left( p\right) \right) c_{5}\left(
p^{+},c_{3}\left( p\right) \right) \right) ^{r}\left\Vert f\right\Vert
_{p\left( \cdot \right) }.
\end{equation*}
\end{corollary}

\begin{definition}
For $p\in P^{Log}\left( \boldsymbol{R}\right) ,$ $f\in L_{p\left( \cdot
\right) }$, $0<\delta <\infty $, $r\in \mathrm{N}_{0}$, we can define
modulus of smoothness as%
\begin{equation*}
\Omega _{r}(f,\delta )_{p(\cdot )}=\Vert (I-T_{\delta })^{r}f\Vert _{p(\cdot
)}
\end{equation*}%
\begin{equation*}
\Omega _{0}(f,\delta )_{p(\cdot )}:=\Vert f\Vert _{p(\cdot )}=:\Omega
_{r}(f,0)_{p(\cdot )}.
\end{equation*}
\end{definition}

\section{Uniform norm estimates}

In this section, let $\Omega \subseteq \boldsymbol{R}$ be a measurable set
and $C\left( \Omega \right) $ be the collection of functions continuous on $%
\Omega $. If $\Omega \neq \boldsymbol{R}$ and $f\in C\left( \Omega \right) $%
, we will extend $f$ to whole $\boldsymbol{R}$ by $"f\left( s\right) \equiv
0 $ whenever $s\notin \Omega $.$"$ when necessary. For $f\in C\left( \Omega
\right) $ and $\delta \geq 0$, we define the modulus of smoothness as 
\begin{equation}
\Omega _{r}(f,\delta )_{C\left( \Omega \right) }:=\Vert (I-T_{\delta
})^{r}f\Vert _{C\left( \Omega \right) },\quad r\in \mathrm{N},  \label{AA}
\end{equation}%
\begin{equation*}
\Omega _{0}(f,\cdot )_{C\left( \Omega \right) }:=\Vert f\Vert _{C\left(
\Omega \right) }
\end{equation*}%
with $T_{\delta }f$ of ($\ast $).

\begin{lemma}
\label{deg1}Let $0\leq \delta <\infty $, $r\in \mathrm{N}$ and $f\in
C^{r}\left( \Omega \right) $. Then%
\begin{equation}
\frac{d^{r}}{dx^{r}}T_{\delta }f\left( x\right) =T_{\delta }\frac{d^{r}}{%
dx^{r}}f\left( x\right) \text{ on }\Omega .  \label{BB}
\end{equation}
\end{lemma}

The following theorem states the main properties of (\ref{AA}).

\begin{theorem}
\label{xv} For $f\in C\left( \Omega \right) $, $0\leq \delta <\infty $, and $%
r\in \mathrm{N}$, the following properties hold.

\begin{enumerate}
\item $\Omega _{r}\left( f,\delta \right) _{C\left( \Omega \right) }$ is
non-negative, non-decreasing function of $\delta $,

\item $\Omega _{r}\left( f,\delta \right) _{C\left( \Omega \right) }$ is
subadditive with respect to $f$,

\item $\left\Vert T_{\delta }f\right\Vert _{C\left( \Omega \right) }\leq
\left\Vert f\right\Vert _{C\left( \Omega \right) },$

\item $\Omega _{r}\left( f,\delta \right) _{C\left( \Omega \right) }\leq
2\Omega _{r-1}\left( f,\delta \right) _{C\left( \Omega \right) }\leq \cdots
\leq 2^{r-1}\Omega _{1}\left( f,\delta \right) _{C\left( \Omega \right)
}\leq 2^{r}\left\Vert f\right\Vert _{C\left( \Omega \right) },\quad $(***)

\item $\Omega _{r}\left( f,\delta \right) _{C\left( \Omega \right) }\leq
2^{-1}\delta \Omega _{r-1}\left( f^{\prime },\delta \right) _{C\left( \Omega
\right) }\leq \cdots \leq 2^{-r}\delta ^{r}\left\Vert f^{\left( r\right)
}\right\Vert _{C\left( \Omega \right) }$,$\quad $if $f\in C^{r}\left( \Omega
\right) .$
\end{enumerate}
\end{theorem}

Let $X$ be a Banach space with a norm $\left\Vert \cdot \right\Vert _{X}$
and $r\in \mathrm{N}$. We define Peetre's \textit{K}-functional for the pair 
$X$ and $W_{X}^{r}$ as follows :%
\begin{equation*}
K_{r}\left( f,\delta ,X\right) _{X}:=\inf\limits_{g\in W_{X}^{r}}\left\{
\left\Vert f-g\right\Vert _{X}+\delta ^{r}\left\Vert g^{\left( r\right)
}\right\Vert _{X}\right\} \text{,\quad }\delta >0.
\end{equation*}%
We set $T_{\delta }^{r}f:=\left( T_{\delta }f\right) ^{r}.$

\begin{lemma}
\label{da1}Let $0\leq \delta <\infty $, $r-1\in \mathrm{N}$, and $f\in
C^{r}\left( \Omega \right) $ be given. Then%
\begin{equation}
\frac{d^{r}}{dx^{r}}T_{\delta }^{r}f\left( x\right) =\frac{d}{dx}T_{\delta }%
\frac{d^{r-1}}{dx^{r-1}}T_{\delta }^{r-1}f\left( x\right) \text{\quad on }%
\Omega .  \label{rrss1}
\end{equation}
\end{lemma}

\begin{lemma}
\label{Ord} (see e.g.\cite[p.177]{devore}) Let $\Omega \subseteq \boldsymbol{%
R}$ be a measurable set, $\delta >0$, $f\in C\left( \Omega \right) $ and $%
\tilde{T}_{\delta }f\left( \cdot \right) =f\left( \cdot +\delta \right) $.
Then, for any $r\in \mathrm{N}$, there holds
\end{lemma}

\begin{equation*}
\frac{1}{r^{r}+2^{r}}\leq \frac{\underset{\left\vert h\right\vert \leq
\delta }{\sup }\left\Vert \left( I-\tilde{T}_{h}\right) ^{r}f\right\Vert
_{C\left( \Omega \right) }}{K_{r}\left( f,\delta ,L_{p}\left( \Omega \right)
\right) _{C\left( \Omega \right) }}\leq 2^{r}\text{.}
\end{equation*}%
Main result of this section is the following theorem.

\begin{theorem}
\label{DI} Let $\Omega \subseteq \boldsymbol{R}$ be a measurable set, $%
0<\delta <\infty $, $f\in C\left( \Omega \right) $, $r\in \mathrm{N}$ and $%
g\in C^{2}\left( \Omega \right) $. Then, the following inequalities%
\begin{align*}
\left\Vert \frac{d}{dx}T_{\delta }f\left( x\right) \right\Vert _{C\left(
\Omega \right) }& \leq \frac{2}{\delta }\left\Vert f\right\Vert _{C\left(
\Omega \right) }, \\
\left\Vert \frac{d^{2}}{dx^{2}}T_{\delta }f\left( x\right) \right\Vert
_{C\left( \Omega \right) }& \leq \frac{2}{\delta }\left\Vert \frac{d}{dx}%
T_{\delta }f\right\Vert _{C\left( \Omega \right) }, \\
\left\Vert g\left( x\right) -T_{\delta }g\left( x\right) +\frac{\delta }{2}%
\frac{d}{dx}g\left( x\right) \right\Vert _{C\left( \Omega \right) }& \leq 
\frac{\delta ^{2}}{6}\left\Vert \frac{d^{2}}{dx^{2}}g\right\Vert _{C\left(
\Omega \right) },
\end{align*}%
\begin{equation}
\left( c_{8}\left( r\right) \right) ^{-1}K_{r}\left( f,\delta ,C\left(
\Omega \right) \right) _{C\left( \Omega \right) }\leq \left\Vert \left(
I-T_{\delta }\right) ^{r}f\right\Vert _{C\left( \Omega \right) }\leq
2^{r}K_{r}\left( f,\delta ,C\left( \Omega \right) \right) _{C\left( \Omega
\right) }\text{,}  \label{eqA}
\end{equation}%
are hold with $c_{8}\left( 1\right) =36$, $c_{8}\left( r\right) =2^{r}\left(
r^{r}+(34)^{r}\right) $ for $r>1$.
\end{theorem}

As a corollary of Theorem \ref{DI} we can state the following result.

\begin{proposition}
\label{DIcor}If $0<h\leq \delta <\infty $ and $f\in C\left( \Omega \right) ,$
then%
\begin{equation}
\left\Vert \left( I-T_{h}\right) f\right\Vert _{C\left( \Omega \right) }\leq
72\left\Vert \left( I-T_{\delta }\right) f\right\Vert _{C\left( \Omega
\right) }.  \label{eq1}
\end{equation}
\end{proposition}

As a corollary of (\ref{eqA}) and Lemma \ref{Ord} we can write

\begin{corollary}
\label{MMM}Let $\Omega \subseteq \boldsymbol{R}$ be a measurable set, $%
\delta >0$, $f\in C\left( \Omega \right) $ and $r\in \mathrm{N}$. Then,

(i) there holds%
\begin{equation*}
1+2^{-r}r^{r}\leq \frac{\underset{\left\vert h\right\vert \leq \delta }{\sup 
}\left\Vert \left( I-\tilde{T}_{h}\right) ^{r}f\right\Vert _{C\left( \Omega
\right) }}{\left\Vert \left( I-T_{\delta }\right) ^{r}f\right\Vert _{C\left(
\Omega \right) }}\leq 2^{r}c_{8}\left( r\right) ,
\end{equation*}

(ii) for $0<\delta _{1}\leq \delta _{2}$, there holds%
\begin{equation*}
\left( 1+2^{-r}r^{r}\right) \Omega _{r}\left( f,\delta _{1}\right) _{C\left(
\Omega \right) }\leq c_{8}\left( r\right) 2^{r}\Omega _{r}\left( f,\delta
_{2}\right) _{C\left( \Omega \right) }.
\end{equation*}
\end{corollary}

\begin{remark}
From Theorem 23.62 of \cite[p.579]{yeh} we have%
\begin{equation}
\lim\limits_{\delta \searrow 0}\Omega _{1}(f,\delta )_{C\left( \boldsymbol{R}%
\right) }=\lim\limits_{\delta \searrow 0}\left\Vert \left( I-T_{\delta
}\right) f\right\Vert _{C\left( \boldsymbol{R}\right) }=0.  \label{sut}
\end{equation}
\end{remark}

\begin{corollary}
\label{yksmk} If $f\in C\left( \boldsymbol{R}\right) $, $0<\delta <\infty $,
and $r\in \mathrm{N}$, then, by (\ref{sut}) and (***),%
\begin{equation*}
\lim\limits_{\delta \searrow 0}\Omega _{r}(f,\delta )_{C\left( \boldsymbol{R}%
\right) }=\lim\limits_{\delta \searrow 0}\left\Vert \left( I-T_{\delta
}\right) ^{r}f\right\Vert _{C\left( \boldsymbol{R}\right) }=0
\end{equation*}%
holds.
\end{corollary}

Let $\mathcal{G}_{\sigma }\left( X\right) $ be the subspace of entire
function of exponential type $\sigma $ that belonging to a Banach space $X$.
The quantity%
\begin{equation}
A_{\sigma }(f)_{X}:=\inf\limits_{g}\{\Vert f-g\Vert _{X}:g\in \mathcal{G}%
_{\sigma }\left( X\right) \}  \label{fd}
\end{equation}%
is called the deviation of the function $f\in X$ from $\mathcal{G}_{\sigma
}\left( X\right) $.

Let $\mathcal{G}_{\sigma ,p(\cdot )}:=\mathcal{G}_{\sigma }\left( L_{p(\cdot
)}\right) $ be the subspace of integral function $f$ of exponential type $%
\sigma $ that belonging to $L_{p(\cdot )}$. The quantity%
\begin{equation*}
A_{\sigma }(f)_{p(\cdot )}:=\inf\limits_{g}\{\Vert f-g\Vert _{p(\cdot
)}:g\in \mathcal{G}_{\sigma ,p(\cdot )}\}
\end{equation*}%
is the deviation of the function $f\in L_{p(\cdot )}$ from $\mathcal{G}%
_{\sigma }$.

\begin{remark}
\label{rmrk} Let $\sigma >0$, $1\leq p\leq \infty $, $f\in L_{p}\left( 
\boldsymbol{R}\right) $,%
\begin{equation*}
\vartheta \left( x\right) :=\frac{2}{\pi }\frac{\sin \left( x/2\right) \sin
(3x/2)}{x^{2}}
\end{equation*}%
and%
\begin{equation*}
J\left( f,\sigma \right) =\sigma \int\nolimits_{\boldsymbol{R}}f\left(
x-u\right) \vartheta \left( \sigma u\right) du
\end{equation*}%
be the dela Val\`{e}e Poussin operator (\cite[definition given in (5.3)]%
{bbsv06}). It is known (see (5.4)-(5.5) of \cite{bbsv06}) that, if $f\in
L_{p}\left( \boldsymbol{R}\right) $, $1\leq p\leq \infty $, then,

(i) $J\left( f,\sigma \right) \in \mathcal{G}_{2\sigma }\left( L_{p}\left( 
\boldsymbol{R}\right) \right) $,

(ii) $J\left( g_{\sigma },\sigma \right) =g_{\sigma }$ for any $g_{\sigma
}\in \mathcal{G}_{\sigma }\left( L_{p}\left( \boldsymbol{R}\right) \right) $,

(iii) $\Vert J\left( f,\sigma \right) \Vert _{L_{p}\left( \boldsymbol{R}%
\right) }\leq \frac{3}{2}\Vert f\Vert _{L_{p}\left( \boldsymbol{R}\right) }$.

(iv) $\left( J\left( f,\sigma \right) \right) ^{\left( r\right) }=J\left(
f^{\left( r\right) },\sigma \right) $ for any $r\in \mathrm{N}$ and $f\in
\left( L_{p}\left( \boldsymbol{R}\right) \right) ^{r}$.

(v) $\Vert J\left( f,\frac{\sigma }{2}\right) -f\Vert _{L_{p}\left( 
\boldsymbol{R}\right) }\rightarrow 0$ (as $\sigma \rightarrow \infty $) and
hence%
\begin{equation*}
\Vert \left( J\left( f,\frac{\sigma }{2}\right) \right) ^{\left( k\right)
}-f^{\left( k\right) }\Vert _{L_{p}\left( \boldsymbol{R}\right) }\rightarrow
0\text{ as }\sigma \rightarrow \infty ,
\end{equation*}%
for $f\in W_{L_{p}\left( \boldsymbol{R}\right) }^{r}$ and $1\leq k\leq r$.
\end{remark}

\begin{corollary}
Let $0<\sigma <\infty $.

(i) If $1\leq p\leq \infty $, $f\in L_{p}\left( \boldsymbol{R}\right) $.
Then, using \textit{(v)} of the last remark, we conclude%
\begin{equation*}
\lim\limits_{\sigma \rightarrow \infty }A_{\sigma }(f)_{L_{p}\left( 
\boldsymbol{R}\right) }=0.
\end{equation*}

(ii) Let $g:\boldsymbol{R}\rightarrow \mathbb{C}$ be bounded on the real
axis $\boldsymbol{R}$. Then (see \cite{B46})%
\begin{equation*}
\lim\limits_{\sigma \rightarrow \infty }A_{\sigma }(g)_{C\left( \boldsymbol{R%
}\right) }=0
\end{equation*}%
if and only if $g$ is uniformly continuous on $\boldsymbol{R}$.
\end{corollary}

\begin{theorem}
\label{jksn} Let $r\in \mathrm{N}$, $\sigma >0$, $\delta \in (0,1)$ and $%
f\in \mathcal{C}(\boldsymbol{R})$. Then, the following Jackson type
inequality%
\begin{equation}
A_{\sigma }\left( f\right) _{\mathcal{C}(\boldsymbol{R})}\leq 5\pi
4^{r-1}c_{8}\left( r\right) \Omega _{r}\left( f,1/\sigma \right) _{\mathcal{C%
}(\boldsymbol{R})}\text{, and}  \label{jaksn}
\end{equation}%
its weak inverse%
\begin{equation}
\Omega _{r}\left( f,\delta \right) _{\mathcal{C}(\boldsymbol{R})}\leq \left(
1+2^{2r-1}\right) 2^{r-1}\delta ^{r}\left( A_{0}\left( f\right) _{\mathcal{C}%
(\boldsymbol{R})}+\int\nolimits_{1/2}^{1/\delta }u^{r-1}A_{u}\left( f\right)
_{\mathcal{C}(\boldsymbol{R})}du\right)  \label{cvrs}
\end{equation}%
are hold$.$
\end{theorem}

We set $\lfloor \sigma \rfloor :=\max \left\{ n\in \mathbb{Z}:n\leq \sigma
\right\} $.

\begin{theorem}
\label{bbsv} Let $r\in \mathrm{N}$, $f\in X_{\mathcal{C}(\boldsymbol{R}%
)}^{r} $ and $\sigma >0$. Then,

(a) (i) there exists (see \cite[Proposition 25]{bbsv06}) a $g_{\sigma }\in 
\mathcal{G}_{\sigma }\left( \mathcal{C}(\boldsymbol{R})\right) $ such that%
\begin{equation*}
A_{\sigma }\left( f\right) _{\mathcal{C}(\boldsymbol{R})}\leq \left\Vert
f-g_{\sigma }\right\Vert _{\mathcal{C}(\boldsymbol{R})}\leq \frac{5\pi }{4}%
\frac{4^{r}}{\sigma ^{r}}\Vert f^{\left( r\right) }\Vert _{\mathcal{C}(%
\boldsymbol{R})}\text{, and}
\end{equation*}

(ii) its weak inverse%
\begin{equation*}
\Vert f^{\left( k\right) }\Vert _{\mathcal{C}(\boldsymbol{R})}\leq \left(
1+2^{2k-1}\right) 2^{k+2}\pi ^{k}c_{8}\left( k\right) \sum\limits_{\nu
=0}^{\infty }\frac{(\nu +1)^{r}}{\nu +1}A_{\nu }\left( f\right) _{\mathcal{C}%
(\boldsymbol{R})}\text{,}
\end{equation*}%
holds whenever $k=1,2,\cdots ,r$ and $\sum\nolimits_{\nu =0}^{\infty }(\nu
+1)^{r-1}A_{\nu }\left( f\right) _{C\left( \boldsymbol{R}\right) }<\infty $.

(b) (i) the following inequality (see \cite[p.397]{II3})%
\begin{equation*}
A_{\sigma }\left( f\right) _{\mathcal{C}(\boldsymbol{R})}\leq \frac{\left(
5\pi \right) ^{r}}{\sigma ^{r}}A_{\sigma }\left( f^{\left( r\right) }\right)
_{\mathcal{C}(\boldsymbol{R})}\text{, and}
\end{equation*}

(ii) its weak inverse%
\begin{equation*}
A_{\sigma }\left( f^{\left( r\right) }\right) _{\mathcal{C}(\boldsymbol{R}%
)}\leq \left\Vert f^{\left( r\right) }-\left( J\left( f^{\left( r\right) },%
\frac{\sigma }{2}\right) \right) \right\Vert _{\mathcal{C}(\boldsymbol{R}%
)}\leq
\end{equation*}%
\begin{equation*}
\leq \left( 1+2^{2r-1}\right) 2^{r+2}\pi ^{r}c_{8}\left( r\right) \left(
A_{\sigma }\left( f\right) _{\mathcal{C}(\boldsymbol{R})}\sum\limits_{k=0}^{%
\lfloor \sigma \rfloor }\frac{k^{r}}{k}+\sum\limits_{\nu =\lfloor \sigma
\rfloor +1}^{\infty }\frac{\left( \nu +1\right) ^{r}}{\nu +1}A_{\nu }\left(
f\right) _{\mathcal{C}(\boldsymbol{R})}\right)
\end{equation*}%
hold when $\sum\nolimits_{\nu =0}^{\infty }(\nu +1)^{r-1}A_{\nu }\left(
f\right) _{\mathcal{C}(\boldsymbol{R})}<\infty .$
\end{theorem}

\begin{theorem}
\label{march} Let $r,k\in \mathrm{N}$, $0<t\leq 1/2$, $0\leq \delta <\infty $
and $f\in \mathcal{C}(\boldsymbol{R})$. Then

(i) there holds%
\begin{equation*}
\Omega _{r+k}\left( f,\delta \right) _{\mathcal{C}(\boldsymbol{R})}\leq
2^{k}\Omega _{r}\left( f,\delta \right) _{\mathcal{C}(\boldsymbol{R})}\text{%
,\quad and}
\end{equation*}

(ii) its weak inverse (Marchaud inequality)%
\begin{equation*}
\Omega _{r}\left( f,t\right) _{\mathcal{C}(\boldsymbol{R})}\leq C_{9}\left(
r,k\right) t^{r}\int_{t}^{1}\frac{\Omega _{r+k}\left( f,u\right) _{\mathcal{C%
}(\boldsymbol{R})}}{u^{r+1}}du
\end{equation*}%
with $C_{9}\left( r,k\right) =10\pi \left( 1+2^{2r-1}\right)
2_{\;}^{2r+3k}c_{8}\left( r+k\right) .$
\end{theorem}

\begin{theorem}
\label{turters} Let $\sigma >0$ and $f\in \mathcal{C}(\boldsymbol{R})$. If $%
\ \sum\nolimits_{\nu =0}^{\infty }(\nu +1)^{k-1}A_{\nu }\left( f\right) _{%
\mathcal{C}(\boldsymbol{R})}<\infty $, holds for some $k\in \mathrm{N}$,
then,

(i) the following Jackson type inequality for derivatives 
\begin{equation*}
A_{\sigma }\left( f\right) _{\mathcal{C}(\boldsymbol{R})}\leq \left( 5\pi
\right) ^{k+1}c_{8}\left( r\right) \sigma ^{-k}\Omega _{r}\left( f^{\left(
k\right) },\sigma ^{-1}\right) _{\mathcal{C}(\boldsymbol{R})}\text{,\quad and%
}
\end{equation*}

(ii) its weak inverse (see Theorem 6.3.4 of \cite[p.343]{II3})%
\begin{equation*}
\Omega _{r}\left( f^{\left( k\right) },\frac{1}{\sigma }\right) _{\mathcal{C}%
(\boldsymbol{R})}\leq 2^{2k+r+1}\left( \frac{1}{\sigma ^{r}}\sum\limits_{\nu
=0}^{\lfloor \sigma \rfloor }\frac{\left( \nu +1\right) ^{r+k}}{\nu +1}%
A_{\nu }\left( f\right) _{\mathcal{C}(\boldsymbol{R})}+\sum\limits_{\nu
=\lfloor \sigma \rfloor +1}^{\infty }\frac{\nu ^{k}}{\nu }A_{\nu }\left(
f\right) _{\mathcal{C}(\boldsymbol{R})}\right)
\end{equation*}%
are hold.
\end{theorem}

\subsection{Proofs of the results of section 2}

\begin{proof}[Proof of Lemma \protect\ref{deg1}]
For $\delta =0$ (\ref{BB}) is obvious. For $0<\delta <\infty $ and $r=1$,
one can find%
\begin{eqnarray}
\frac{d}{dx}T_{\delta }f(x) &=&\frac{d}{dx}\left( \frac{1}{\delta }%
\int\nolimits_{0}^{\delta }f\left( x+t\right) dt\right) =\frac{d}{dx}\left( 
\frac{1}{\delta }\int\nolimits_{x}^{x+\delta }f\left( \tau \right) d\tau
\right)  \label{CC} \\
&=&\frac{1}{\delta }\int\nolimits_{x}^{x+\delta }\frac{d}{dx}f\left( \tau
\right) d\tau =T_{\delta }\frac{d}{dx}f(x).  \notag
\end{eqnarray}

For $r>1$, \ (\ref{BB}) follows from (\ref{CC}).
\end{proof}

\begin{proof}[Proof of Theorem \protect\ref{xv}]
(1)-(3) is known. (4) is seen from binomial expansion. To prove (5) it is
sufficient to note inequality (see \cite{AG}) 
\begin{equation*}
\left\Vert \left( I-T_{\delta }\right) f\right\Vert _{C\left( \Omega \right)
}\leq 2^{-1}\delta \left\Vert f^{\prime }\right\Vert _{C\left( \Omega
\right) }\text{,}\quad \delta >0
\end{equation*}%
for $f\in C^{1}\left( \Omega \right) $. Then%
\begin{equation*}
\left\Vert \left( I-T_{\delta }\right) ^{r}f\right\Vert _{C\left( \Omega
\right) }\leq 2^{-1}\delta \left\Vert \left( I-T_{\delta }\right)
^{r-1}f^{\prime }\right\Vert _{C\left( \Omega \right) }\leq \cdots \leq
2^{-r}\delta ^{r}\left\Vert f^{(r)}\right\Vert _{C\left( \Omega \right) }
\end{equation*}%
for $f\in C^{r}\left( \Omega \right) $, because%
\begin{equation*}
\left[ \left( I-T_{\delta }\right) ^{r}f\right] ^{\prime }=\left(
I-T_{\delta }\right) ^{r}f^{\prime }.
\end{equation*}
\end{proof}

\begin{proof}[Proof of Lemma \protect\ref{da1}]
For $r=2$, by Lemma \ref{deg1},%
\begin{eqnarray*}
\frac{d^{2}}{dx^{2}}T_{\delta }^{2}f &=&\frac{d}{dx}\frac{d}{dx}T_{\delta
}T_{\delta }f=\frac{d}{dx}\frac{d}{dx}T_{\delta }\Psi \qquad \left[ \Psi
:=T_{\delta }f\right] \\
&=&\frac{d}{dx}T_{\delta }\frac{d}{dx}\Psi =\frac{d}{dx}T_{\delta }\frac{d}{%
dx}T_{\delta }f
\end{eqnarray*}%
and the result (\ref{rrss1}) follows. For $r=3$, by Lemma \ref{deg1},%
\begin{equation*}
\frac{d^{3}}{dx^{3}}T_{\delta }^{3}f=\frac{d}{dx}\frac{d^{2}}{dx^{2}}%
T_{\delta }^{2}R_{v}f=\frac{d}{dx}\frac{d^{2}}{dx^{2}}T_{\delta }^{2}\Psi =%
\frac{d}{dx}\frac{d}{dx}T_{\delta }\frac{d}{dx}T_{\delta }\Psi
\end{equation*}%
\begin{equation*}
=\frac{d}{dx}\frac{d}{dx}T_{\delta }\frac{d}{dx}T_{\delta }{}^{2}f=\frac{d}{%
dx}T_{\delta }\frac{d}{dx}\frac{d}{dx}T_{\delta }{}^{2}f=\frac{d}{dx}%
T_{\delta }\frac{d^{2}}{dx^{2}}T_{\delta }^{2}f
\end{equation*}%
and (\ref{rrss1}) holds. Let (\ref{rrss1}) holds for $k\in \mathrm{N}$:%
\begin{equation}
\frac{d^{k}}{dx^{k}}T_{\delta }^{k}f=\frac{d}{dx}T_{\delta }\frac{d^{k-1}}{%
dx^{k-1}}T_{\delta }^{k-1}f.  \label{bula}
\end{equation}%
Then, for $k+1$, (\ref{bula}) and Lemma \ref{deg1} implies that%
\begin{equation*}
\frac{d^{k+1}}{dx^{k+1}}T_{\delta }^{k+1}f=\frac{d}{dx}\frac{d^{k}}{dx^{k}}%
T_{\delta }^{k}T_{\delta }f=\frac{d}{dx}\frac{d^{k}}{dx^{k}}T_{\delta
}^{k}\Psi =\frac{d}{dx}\frac{d}{dx}T_{\delta }\frac{d^{k-1}}{dx^{k-1}}%
T_{\delta }^{k-1}\Psi
\end{equation*}%
\begin{equation*}
=\frac{d}{dx}\frac{d}{dx}T_{\delta }\frac{d^{k-1}}{dx^{k-1}}T_{\delta }^{k}f=%
\frac{d}{dx}T_{\delta }\frac{d}{dx}\frac{d^{k-1}}{dx^{k-1}}T_{\delta }^{k}f=%
\frac{d}{dx}T_{\delta }\frac{d^{k}}{dx^{k}}T_{\delta }^{k}f.
\end{equation*}
\end{proof}

\begin{proof}[Proof of Theorem \protect\ref{DI}]
For $f\in C\left( \Omega \right) $ we have%
\begin{eqnarray}
\left\Vert \frac{d}{dx}T_{\delta }f\left( x\right) \right\Vert _{C\left(
\Omega \right) } &=&\left\Vert \frac{d}{dx}\frac{1}{\delta }%
\int\nolimits_{0}^{\delta }f\left( x+t\right) dt\right\Vert _{C\left( \Omega
\right) }=  \notag \\
\left\Vert \frac{1}{\delta }\frac{d}{dx}\int\nolimits_{x}^{x+\delta }f\left(
\tau \right) d\tau \right\Vert _{C\left( \Omega \right) } &=&\left\Vert 
\frac{1}{\delta }\left( f\left( x+\delta \right) -f\left( x\right) \right)
\right\Vert _{C\left( \Omega \right) }\leq \frac{2}{\delta }\left\Vert
f\right\Vert _{C\left( \Omega \right) }.  \label{iki}
\end{eqnarray}%
Inequality (\ref{iki}) also implies%
\begin{equation*}
\left\Vert \left( \frac{d}{dx}\right) ^{2}T_{\delta }f\left( x\right)
\right\Vert _{C\left( \Omega \right) }\leq \frac{2}{\delta }\left\Vert \frac{%
d}{dx}T_{\delta }f\right\Vert _{C\left( \Omega \right) }
\end{equation*}%
for $f\in C\left( \Omega \right) $. If $f\in C^{2}\left( \Omega \right) $
one can get%
\begin{equation}
\left\Vert f\left( x\right) -T_{\delta }f\left( x\right) +\frac{\delta }{2}%
\frac{d}{dx}f\left( x\right) \right\Vert _{C\left( \Omega \right) }\leq 
\frac{\delta ^{2}}{6}\left\Vert \frac{d^{2}}{dx^{2}}f\right\Vert _{C\left(
\Omega \right) }.  \label{vor}
\end{equation}%
To obtain (\ref{vor}) we will use the Taylor formula%
\begin{equation*}
f\left( x+t\right) =f(x)+t\frac{d}{dx}f(x)+\frac{t^{2}}{2}\frac{d^{2}}{dx^{2}%
}f(\xi )
\end{equation*}%
for some $\xi \leq \left[ x,x+t\right] $. Then integrating the last equation
with respect to $t$%
\begin{align*}
\frac{1}{\delta }\int\nolimits_{0}^{\delta }f\left( x+t\right) dt& =f(x)+%
\frac{1}{\delta }\int\nolimits_{0}^{\delta }tdt\frac{d}{dx}f(x)+\frac{1}{2}%
\frac{1}{\delta }\int\nolimits_{0}^{\delta }t^{2}dt\frac{d^{2}}{dx^{2}}f(\xi
), \\
T_{\delta }f\left( x\right) & =f(x)+\frac{\delta }{2}\frac{d}{dx}f\left(
x\right) +\frac{\delta ^{2}}{6}\frac{d^{2}}{dx^{2}}f(\xi )
\end{align*}%
and (\ref{vor}) holds.

Now (\ref{iki}) and (\ref{vor}) imply that%
\begin{equation}
\left( 1/36\right) K_{1}\left( f,\delta ,C\left( \Omega \right) \right)
_{C\left( \Omega \right) }\leq \left\Vert \left( \mathbb{I}-T_{\delta
}\right) f\right\Vert _{C\left( \Omega \right) }\leq 2K_{1}\left( f,\delta
,C\left( \Omega \right) \right) _{C\left( \Omega \right) }.  \label{DI3}
\end{equation}%
Firstly, let us prove the right hand side of (\ref{DI3}). For any $g\in
C^{1}\left( \Omega \right) $%
\begin{align*}
\left\Vert f-T_{\delta }f\right\Vert _{C\left( \Omega \right) }& \leq
\left\Vert f-g\right\Vert _{C\left( \Omega \right) }+\left\Vert g-T_{\delta
}g\right\Vert _{C\left( \Omega \right) }+\left\Vert T_{\delta }\left(
g-f\right) \right\Vert _{C\left( \Omega \right) } \\
& \leq 2\left\Vert f-g\right\Vert _{C\left( \Omega \right) }+\frac{\delta }{2%
}\left\Vert g^{\prime }\right\Vert _{C\left( \Omega \right) }\leq
2K_{1}\left( f,\delta ,C\left( \Omega \right) \right) _{C\left( \Omega
\right) }.
\end{align*}

For the left hand side of inequality (\ref{DI3}) we need inequalities%
\begin{align}
\left\Vert f-T_{\delta }^{2}f\right\Vert _{C\left( \boldsymbol{R}\right) }&
\leq 2\left\Vert f-T_{\delta }f\right\Vert _{C\left( \boldsymbol{R}\right) },
\label{uc} \\
\delta \left\Vert \left( \frac{d}{dx}\right) ^{2}T_{\delta }^{2}f\right\Vert
_{C\left( \boldsymbol{R}\right) }& \leq 34\left\Vert f-T_{\delta
}f\right\Vert _{C\left( \boldsymbol{R}\right) }.  \label{dort}
\end{align}

First we prove (\ref{uc}). Then%
\begin{align*}
\left\Vert f-T_{\delta }^{2}f\right\Vert _{C\left( \Omega \right) }& \leq
\left\Vert f-T_{\delta }f\right\Vert _{C\left( \Omega \right) }+\left\Vert
T_{\delta }f-T_{\delta }T_{\delta }f\right\Vert _{C\left( \Omega \right) } \\
& \leq 2\left\Vert f-T_{\delta }f\right\Vert _{C\left( \Omega \right) }.
\end{align*}%
Now we consider inequality (\ref{dort}). In (\ref{vor}) we replace $f$ by $%
T_{\delta }^{2}f$ and obtain%
\begin{equation*}
\left\Vert T_{\delta }^{2}f\left( x\right) -T_{\delta }T_{\delta
}^{2}f\left( x\right) +\frac{\delta }{2}\frac{d}{dx}T_{\delta }^{2}f\left(
x\right) \right\Vert _{C\left( \Omega \right) }\leq \frac{\delta ^{2}}{6}%
\left\Vert \frac{d^{2}}{dx^{2}}T_{\delta }^{2}f\right\Vert _{C\left( \Omega
\right) }.
\end{equation*}%
On the other hand, by (\ref{iki}),%
\begin{align*}
\left\Vert \frac{d^{2}}{dx^{2}}T_{\delta }^{2}f\right\Vert _{C\left( \Omega
\right) }& \leq \frac{2}{\delta }\left\Vert \frac{d}{dx}T_{\delta
}f\right\Vert _{C\left( \Omega \right) } \\
& \leq \frac{2}{\delta }\left\{ \left\Vert \frac{d}{dx}T_{\delta
}^{2}f\right\Vert _{C\left( \Omega \right) }+\left\Vert \frac{d}{dx}%
T_{\delta }\left( T_{\delta }f-f\right) \right\Vert _{C\left( \Omega \right)
}\right\} \\
& \leq \frac{2}{\delta }\left\Vert \frac{d}{dx}T_{\delta }^{2}f\right\Vert
_{C\left( \Omega \right) }+\frac{4}{\delta ^{2}}\left\Vert T_{\delta
}f-f\right\Vert _{C\left( \Omega \right) }.
\end{align*}%
Hence,%
\begin{align*}
\frac{\delta }{2}\left\Vert \frac{d}{dx}T_{\delta }^{2}f\right\Vert
_{C\left( \Omega \right) }& \leq \left\Vert T_{\delta }^{2}f-T_{\delta
}T_{\delta }^{2}f-\frac{\delta }{2}\frac{d}{dx}T_{\delta }^{2}f\right\Vert
_{C\left( \Omega \right) }+\left\Vert T_{\delta }^{2}f-T_{\delta }T_{\delta
}^{2}f\right\Vert _{C\left( \Omega \right) } \\
& \leq \frac{\delta ^{2}}{6}\left\Vert \frac{d^{2}}{dx^{2}}T_{\delta
}^{2}f\right\Vert _{C\left( \Omega \right) }+\left\Vert T_{\delta
}^{2}f-T_{\delta }T_{\delta }^{2}f\right\Vert _{C\left( \Omega \right) }
\end{align*}%
\begin{eqnarray*}
&\leq &\frac{\delta ^{2}}{6}\frac{2}{\delta }\left\{ \left\Vert \frac{d}{dx}%
T_{\delta }^{2}f\right\Vert _{C\left( \Omega \right) }+\frac{2}{\delta }%
\left\Vert T_{\delta }f-f\right\Vert _{C\left( \Omega \right) }\right\}
+\left\Vert T_{\delta }^{2}f-f\right\Vert _{C\left( \Omega \right) } \\
&&+\left\Vert T_{\delta }\left( T_{\delta }^{2}f-f\right) \right\Vert
_{C\left( \Omega \right) }+\left\Vert T_{\delta }f-f\right\Vert _{C\left(
\Omega \right) }.
\end{eqnarray*}%
Then%
\begin{align*}
\frac{\delta }{6}\left\Vert \frac{d}{dx}T_{\delta }^{2}f\right\Vert
_{C\left( \Omega \right) }& \leq \frac{17}{3}\left\Vert T_{\delta
}f-f\right\Vert _{C\left( \Omega \right) }, \\
\delta \left\Vert \frac{d}{dx}T_{\delta }^{2}f\right\Vert _{C\left( \Omega
\right) }& \leq 34\left\Vert T_{\delta }f-f\right\Vert _{C\left( \Omega
\right) }.
\end{align*}%
To finish proof of the left hand side of inequality (\ref{eqA}) with $r=1$,
we proceed as%
\begin{align*}
K_{1}\left( f,\delta ,C\left( \Omega \right) \right) _{C\left( \Omega
\right) }& \leq \left\Vert f-T_{\delta }^{2}f\right\Vert _{C\left( \Omega
\right) }+\delta \left\Vert \frac{d}{dx}T_{\delta }^{2}f\right\Vert
_{C\left( \Omega \right) } \\
& \leq 36\left\Vert T_{\delta }f-f\right\Vert _{C\left( \Omega \right) }.
\end{align*}

The proof of (\ref{eqA}) with $r=1$ now completed.

Let $r>1$ be a natural number and we define%
\begin{equation*}
g\left( \cdot \right) =\sum\limits_{l=1}^{r}\left( -1\right) ^{l-1}\binom{r}{%
l}T_{\delta }^{2rl}f\left( \cdot \right) .
\end{equation*}%
Then,%
\begin{equation*}
\left\Vert f-g\right\Vert _{C\left( \Omega \right) }=\left\Vert \left(
I-T_{\delta }^{2r}\right) ^{r}f\right\Vert _{C\left( \Omega \right) }\leq
(2r)^{r}\left\Vert \left( I-T_{\delta }\right) ^{r}f\right\Vert _{C\left(
\Omega \right) }.
\end{equation*}%
On the other hand,%
\begin{align*}
\delta ^{r}\left\Vert \frac{d^{r}}{dx^{r}}T_{\delta }^{2r}f\right\Vert
_{C\left( \Omega \right) }& =\delta ^{r-1}\delta \left\Vert \frac{d}{dx}%
T_{\delta }^{2}\left( \frac{d^{r-1}}{dx^{r-1}}\right) T_{\delta
}^{2r-2}f\right\Vert _{C\left( \Omega \right) } \\
& \leq 34\delta ^{r-1}\left\Vert \left( I-T_{\delta }\right) \frac{d^{r-1}}{%
dx^{r-1}}T_{\delta }^{2r-2}f\right\Vert _{C\left( \Omega \right) } \\
& \leq \left( 34\right) ^{2}\delta ^{r-2}\left\Vert \left( I-T_{\delta
}\right) ^{2}\frac{d^{r-2}}{dx^{r-2}}T_{\delta }^{2r-4}f\right\Vert
_{C\left( \Omega \right) } \\
& \leq \cdots \leq (34)^{r}\left\Vert \left( I-T_{\delta }\right)
^{r}f\right\Vert _{C\left( \Omega \right) }.
\end{align*}%
Then%
\begin{eqnarray*}
\delta ^{r}\left\Vert \frac{d^{r}}{dx^{r}}T_{\delta }^{2rl}f\right\Vert
_{C\left( \Omega \right) } &\leq &(34)^{r}\left\Vert \left( I-T_{\delta
}\right) ^{r}T_{\delta }^{2r\left( l-1\right) }f\right\Vert _{C\left( \Omega
\right) } \\
&=&(34)^{r}\left\Vert T_{\delta }^{2r\left( l-1\right) }\left( I-T_{\delta
}\right) ^{r}f\right\Vert _{C\left( \Omega \right) }\leq (34)^{r}\left\Vert
\left( I-T_{\delta }\right) ^{r}f\right\Vert _{C\left( \Omega \right) }.
\end{eqnarray*}%
Using the last inequality we find%
\begin{eqnarray*}
\delta ^{r}\left\Vert \frac{d^{r}}{dx^{r}}g\right\Vert _{C\left( \Omega
\right) } &=&\delta ^{r}\left\Vert \frac{d^{r}}{dx^{r}}\sum\limits_{l=1}^{r}%
\left( -1\right) ^{l-1}\binom{r}{l}T_{\delta }^{2rl}f\right\Vert _{C\left(
\Omega \right) } \\
&=&\delta ^{r}\left\Vert \sum\limits_{l=1}^{r}\left( -1\right) ^{l-1}\binom{r%
}{l}\frac{d^{r}}{dx^{r}}T_{\delta }^{2rl}f\right\Vert _{C\left( \Omega
\right) } \\
&\leq &\sum\limits_{l=1}^{r}\left\vert \binom{r}{l}\right\vert \delta
^{r}\left\Vert \frac{d^{r}}{dx^{r}}T_{\delta }^{2rl}f\right\Vert _{C\left(
\Omega \right) }\leq 2^{r}(34)^{r}\left\Vert \left( I-T_{\delta }\right)
^{r}f\right\Vert _{C\left( \Omega \right) }
\end{eqnarray*}%
and%
\begin{eqnarray*}
K_{r}\left( f,\delta ,C\left( \Omega \right) \right) _{C\left( \Omega
\right) } &\leq &\left\Vert f-g\right\Vert _{C\left( \Omega \right) }+\delta
^{r}\left\Vert \frac{d^{r}}{dx^{r}}g\right\Vert _{C\left( \Omega \right) } \\
&\leq &2^{r}\left( r^{r}+(34)^{r}\right) \left\Vert \left( I-T_{\delta
}\right) ^{r}f\right\Vert _{C\left( \Omega \right) }.
\end{eqnarray*}%
For the opposite direction of the last inequality, when $g\in W_{p\left(
\cdot \right) }^{r}$,%
\begin{align}
\Omega _{r}\left( f,\delta \right) _{C\left( \Omega \right) }& \leq
2^{r}\left\Vert f-g\right\Vert _{C\left( \Omega \right) }+\Omega _{r}\left(
g,\delta \right) _{C\left( \Omega \right) }  \notag \\
& \leq 2^{r}\left\Vert f-g\right\Vert _{C\left( \Omega \right)
}+2^{-r}\delta ^{r}\left\Vert g^{\left( r\right) }\right\Vert _{C\left(
\Omega \right) },  \label{mn}
\end{align}%
and taking infimum on $g\in W_{p\left( \cdot \right) }^{r}$ in (\ref{mn}) we
get%
\begin{equation*}
\Omega _{r}\left( f,\delta \right) _{C\left( \Omega \right) }\leq
2^{r}K_{r}\left( f,\delta ,C\left( \Omega \right) \right) _{C\left( \Omega
\right) }.
\end{equation*}
\end{proof}

\begin{proof}[Proof of Proposition \protect\ref{DIcor}]
Let $f\in C\left( \Omega \right) $. Then%
\begin{eqnarray*}
\left\Vert \left( I-T_{h}\right) f\right\Vert _{C\left( \Omega \right) }
&\leq &2K_{1}\left( f,h,C\left( \Omega \right) \right) _{C\left( \Omega
\right) } \\
&\leq &2K_{1}\left( f,\delta ,C\left( \Omega \right) \right) _{C\left(
\Omega \right) }\leq 72\left\Vert \left( I-T_{\delta }\right) f\right\Vert
_{C\left( \Omega \right) }.
\end{eqnarray*}
\end{proof}

\begin{proof}[Proof of Theorem \protect\ref{jksn}]
(i) We consider Jackson type inequality (\ref{jaksn}). For any $g\in X_{%
\mathcal{C}(\boldsymbol{R})}^{r}$ we have%
\begin{equation*}
A_{\sigma }\left( f\right) _{\mathcal{C}(\boldsymbol{R})}\leq A_{\sigma
}\left( f-g\right) _{\mathcal{C}(\boldsymbol{R})}+A_{\sigma }\left( g\right)
_{\mathcal{C}(\boldsymbol{R})}
\end{equation*}%
\begin{equation*}
\leq \left\Vert f-g\right\Vert _{\mathcal{C}(\boldsymbol{R})}+\frac{5\pi }{4}%
\frac{4^{r}}{\sigma ^{r}}\left\Vert \frac{d^{r}}{dx^{r}}g\right\Vert _{%
\mathcal{C}(\boldsymbol{R})}.
\end{equation*}%
Taking infimum on $g\in X_{\mathcal{C}(\boldsymbol{R})}^{r}$ in the last
inequality we have 
\begin{equation*}
A_{\sigma }\left( f\right) _{\mathcal{C}(\boldsymbol{R})}\leq \frac{5\pi
4^{r}}{4}K_{r}\left( f,\frac{1}{\sigma },\mathcal{C}(\boldsymbol{R})\right)
_{\mathcal{C}(\boldsymbol{R})}\leq \frac{5\pi }{4}c_{8}\left( r\right)
4^{r}\left\Vert \left( I-T_{\frac{1}{\sigma }}\right) ^{r}f\right\Vert _{%
\mathcal{C}(\boldsymbol{R})}.
\end{equation*}

(ii) We give the proof of inverse estimate (\ref{cvrs}). Let $\sigma >0$ and 
$g_{\sigma }\in \mathcal{G}_{\sigma }\left( \mathcal{C}(\boldsymbol{R}%
)\right) $ be the best approximating IFFD of $f\in \mathcal{C}(\boldsymbol{R}%
)$. Suppose that $r\in \mathrm{N}$, $0<\delta <1$. Then, there exists a $%
m\in \mathrm{N}$ such that $\lfloor 1/\delta \rfloor $ $=2^{m-1}.$ Hence, $%
2^{m-1}\leq 1/\delta <2^{m}.$ Now we have%
\begin{eqnarray*}
\Omega _{r}\left( f,\delta \right) _{\mathcal{C}(\boldsymbol{R})} &\leq
&\Omega _{r}\left( f-g_{2^{m}},\delta \right) _{\mathcal{C}(\boldsymbol{R}%
)}+\Omega _{r}\left( g_{2^{m}},\delta \right) _{\mathcal{C}(\boldsymbol{R})}
\\
&\leq &2^{r}A_{2^{m}}\left( f\right) _{\mathcal{C}(\boldsymbol{R}%
)}+2^{-r}\delta ^{r}\left\Vert \frac{d^{r}}{dx^{r}}g_{2^{m}}\right\Vert _{%
\mathcal{C}(\boldsymbol{R})}.
\end{eqnarray*}%
On the other hand%
\begin{equation*}
\left\Vert \frac{d^{r}}{dx^{r}}g_{2^{m}}\right\Vert _{\mathcal{C}(%
\boldsymbol{R})}=\left\Vert \sum_{\gamma =1}^{m}\left( \frac{d^{r}}{dx^{r}}%
g_{2^{\gamma }}-\frac{d^{r}}{dx^{r}}g_{2^{\gamma -1}}\right) +\left( \frac{%
d^{r}}{dx^{r}}g_{1}-\frac{d^{r}}{dx^{r}}g_{0}\right) \right\Vert _{\mathcal{C%
}(\boldsymbol{R})}
\end{equation*}%
\begin{equation*}
\leq \sum_{\gamma =1}^{m}2^{\gamma r}\left\Vert g_{2^{\gamma }}-g_{2^{\gamma
-1}}\right\Vert _{\mathcal{C}(\boldsymbol{R})}+\left\Vert
g_{1}-g_{0}\right\Vert _{\mathcal{C}(\boldsymbol{R})}
\end{equation*}%
\begin{equation*}
\leq A_{0}\left( f\right) _{\mathcal{C}(\boldsymbol{R})}+A_{1}\left(
f\right) _{\mathcal{C}(\boldsymbol{R})}+\sum_{\gamma =1}^{m}2^{\gamma
r}\left( A_{2^{\gamma }}\left( f\right) _{\mathcal{C}(\boldsymbol{R}%
)}+A_{2^{\gamma -1}}\left( f\right) _{\mathcal{C}(\boldsymbol{R})}\right)
\end{equation*}%
\begin{equation*}
\leq A_{0}\left( f\right) _{\mathcal{C}(\boldsymbol{R})}+2^{r}A_{1}\left(
f\right) _{\mathcal{C}(\boldsymbol{R})}+2\sum_{\gamma =1}^{m}2^{\gamma
r}A_{2^{\gamma -1}}\left( f\right) _{\mathcal{C}(\boldsymbol{R})}
\end{equation*}%
\begin{equation*}
\leq 2\left( A_{0}\left( f\right) _{\mathcal{C}(\boldsymbol{R}%
)}+\sum_{\gamma =1}^{m}2^{\gamma r}A_{2^{\gamma -1}}\left( f\right) _{%
\mathcal{C}(\boldsymbol{R})}\right) .
\end{equation*}%
Then

\begin{equation*}
\frac{\delta ^{r}}{2^{r}}\left\Vert \frac{d^{r}}{dx^{r}}g_{2^{m}}\right\Vert
_{\mathcal{C}(\boldsymbol{R})}\leq \frac{2}{2^{r}}\delta ^{r}\left(
A_{0}\left( f\right) _{\mathcal{C}(\boldsymbol{R})}+\sum_{\gamma
=1}^{m}2^{\gamma r}A_{q^{\gamma -1}}\left( f\right) _{\mathcal{C}(%
\boldsymbol{R})}\right) .
\end{equation*}%
Hence%
\begin{equation*}
\Omega _{r}\left( f,\delta \right) _{C(\boldsymbol{R})}\leq \frac{2^{(m+1)r}%
}{2^{mr}}A_{2^{m}}\left( f\right) _{\mathcal{C}(\boldsymbol{R})}+\frac{2}{%
2^{r}}\delta ^{r}\left( A_{0}\left( f\right) _{\mathcal{C}(\boldsymbol{R}%
)}+\sum_{\gamma =1}^{m}2^{\gamma r}A_{q^{\gamma -1}}\left( f\right) _{%
\mathcal{C}(\boldsymbol{R})}\right)
\end{equation*}%
\begin{equation*}
\leq \left( 1+2^{2r-1}\right) 2^{1-r}2^{2r}\delta ^{r}\left( A_{0}\left(
f\right) _{\mathcal{C}(\boldsymbol{R})}+\sum_{\gamma
=1}^{m}\int\limits_{2^{\gamma -2}}^{2^{\gamma -1}}u^{r-1}A_{u}\left(
f\right) _{\mathcal{C}(\boldsymbol{R})}du\right)
\end{equation*}%
\begin{equation*}
\leq \left( 1+2^{2r-1}\right) 2^{r-1}\delta ^{r}\left( A_{0}\left( f\right)
_{\mathcal{C}(\boldsymbol{R})}+\int_{1/2}^{2^{m-1}}u^{r-1}A_{u}\left(
f\right) _{\mathcal{C}(\boldsymbol{R})}du\right)
\end{equation*}%
\begin{equation*}
\leq \left( 1+2^{2r-1}\right) 2^{r-1}\delta ^{r}\left( A_{0}\left( f\right)
_{\mathcal{C}(\boldsymbol{R})}+\int_{1/2}^{1/\delta }u^{r-1}A_{u}\left(
f\right) _{\mathcal{C}(\boldsymbol{R})}du\right) .
\end{equation*}
\end{proof}

\begin{proof}[Proof of Theorem \protect\ref{bbsv}]
Results a) (i) \ and b) (i) are known. Let us consider a) (ii). Suppose that 
$\sum\limits_{\nu =0}^{\infty }\frac{(\nu +1)^{r}}{\nu +1}A_{\nu }\left(
f\right) _{\mathcal{C}(\boldsymbol{R})}<\infty $ and $k\in \left\{
1,2,\cdots ,r\right\} $. Then, using Nikolskii inequality, one gets 
\begin{eqnarray*}
\Vert f^{\left( k\right) }\Vert _{\mathcal{C}(\boldsymbol{R})}
&=&\lim\limits_{\sigma \rightarrow \infty }\Vert J\left( f^{\left( k\right)
},\frac{\sigma }{2}\right) \Vert _{\mathcal{C}(\boldsymbol{R}%
)}=\lim\limits_{\sigma \rightarrow \infty }\Vert \left( J\left( f,\frac{%
\sigma }{2}\right) \right) ^{\left( k\right) }\Vert _{\mathcal{C}(%
\boldsymbol{R})} \\
&\leq &\frac{\pi ^{k}}{2^{k}}\frac{\underset{\left\vert h\right\vert \leq
\delta }{\sup }\left\Vert \left( I-\tilde{T}_{h}\right) ^{k}\left( J\left( f,%
\frac{\sigma }{2}\right) \right) \right\Vert _{\mathcal{C}(\boldsymbol{R})}}{%
\delta ^{k}}\leq \frac{\pi ^{k}}{2^{k}}\frac{2^{k}c_{8}\left( k\right)
\Omega _{k}\left( J\left( f,\frac{\sigma }{2}\right) ,\delta \right) _{%
\mathcal{C}(\boldsymbol{R})}}{\delta ^{k}} \\
&\leq &\left( 1+2^{2k-1}\right) 2^{k+2}\pi ^{k}c_{8}\left( k\right)
\sum\limits_{\nu =0}^{\lfloor 1/\delta \rfloor }\frac{(\nu +1)^{k}}{\nu +1}%
A_{\nu }\left( J\left( f,\frac{\sigma }{2}\right) \right) _{\mathcal{C}(%
\boldsymbol{R})} \\
&\leq &\left( 1+2^{2k-1}\right) 2^{k+2}\pi ^{k}c_{8}\left( k\right)
\sum\limits_{\nu =0}^{\infty }\frac{(\nu +1)^{r}}{\nu +1}A_{\nu }\left(
f\right) _{\mathcal{C}(\boldsymbol{R})}.
\end{eqnarray*}

Note that (ii) b) is follow from (i) b).
\end{proof}

\begin{proof}[Proof of Theorem \protect\ref{march}]
(i) follows from properties of modulus of smoothness. We consider Marchaud
type inequality (ii). CLet $0<t<1/2$. Assume that $2^{m-1}\leq \frac{1}{t}%
<2^{m}$ for some $m\in \mathbb{N}$. Then

\begin{equation*}
\Omega _{r}(f,t)_{\mathcal{C}(\boldsymbol{R})}\leq \left( 1+2^{2r-1}\right)
2^{1-r}t^{r}\left( \sum\limits_{\nu =1}^{m}2^{\nu r}A_{2^{\nu -1}}\left(
f\right) _{\mathcal{C}(\boldsymbol{R})}+A_{0}\left( f\right) _{\mathcal{C}(%
\boldsymbol{R})}\right)
\end{equation*}%
\begin{equation*}
\leq \frac{5\pi }{2}\left( 1+2^{2r-1}\right) 2^{r+2k}c_{8}\left( r+k\right)
t^{r}\left( A_{0}\left( f\right) _{\mathcal{C}(\boldsymbol{R}%
)}+\sum\limits_{\nu =1}^{m}2^{\nu r}\Omega _{k+r}(f,\frac{1}{2^{\nu }})_{%
\mathcal{C}(\boldsymbol{R})}\right)
\end{equation*}%
\begin{equation*}
\leq \frac{5\pi }{2}\left( 1+2^{2r-1}\right) 2^{2r+3k}c_{8}\left( r+k\right)
t^{r}\left( \Omega _{k+r}(f,\frac{1}{2})_{\mathcal{C}(\boldsymbol{R}%
)}+\sum\limits_{\nu =1}^{m}\int\limits_{2^{-v}}^{2^{-v+1}}\frac{\Omega
_{k+r}(f,u)_{\mathcal{C}(\boldsymbol{R})}}{u^{r+1}}du\right)
\end{equation*}%
\begin{equation*}
\leq \frac{5\pi }{2}\left( 1+2^{2r-1}\right) 2^{2r+3k}c_{8}\left( r+k\right)
t^{r}\left( \Omega _{k+r}(f,\frac{1}{2})_{\mathcal{C}(\boldsymbol{R}%
)}+\int\limits_{2^{-1}}^{2^{-m+1}}\frac{\Omega _{k+r}(f,u)_{\mathcal{C}(%
\boldsymbol{R})}}{u^{r+1}}du\right)
\end{equation*}%
\begin{equation*}
\leq 5\pi \left( 1\text{+}2^{2r-1}\right) 2^{2r+3k}c_{8}\left( r\text{+}%
k\right) t^{r}\left( \int\limits_{1/2}^{1}\frac{\Omega _{k+r}(f,u)_{\mathcal{%
C}(\boldsymbol{R})}}{u^{r+1}}du\text{+}\int\limits_{t}^{1}\frac{\Omega
_{k+r}(f,u)_{\mathcal{C}(\boldsymbol{R})}}{u^{r+1}}du\right)
\end{equation*}%
\begin{equation*}
\leq 10\pi \left( 1+2^{2r-1}\right) 2^{2r+3k}c_{8}\left( r+k\right)
t^{k}\int\limits_{t}^{1}\frac{\Omega _{k+r}(f,u)_{\mathcal{C}(\boldsymbol{R}%
)}}{u^{r+1}}du
\end{equation*}
\end{proof}

Using this section's estimates and Transference Result Theorem \ref{tra}, in
the next section we will give several results on difference operator $%
\left\Vert \left( I-T_{\delta }\right) ^{r}f\right\Vert _{p\left( \cdot
\right) }$ and approximation by IFFD in $L_{p\left( \cdot \right) }.$

\section{Applications on Difference operator and Approximation}

\begin{notation}
Since the $48c_{7}\left( c_{3}\left( p\right) \right) c_{5}\left(
p^{+},c_{3}\left( p\right) \right) $ of (\ref{hurr}) will be used very
frequently in the next parts we will set $c_{10}$:=$c_{10}\left(
p^{+},c_{3}\left( p\right) \right) $:=$48c_{7}\left( c_{3}\left( p\right)
\right) c_{5}\left( p^{+},c_{3}\left( p\right) \right) .$
\end{notation}

\begin{lemma}
\label{lemma1} Let $p\in P^{Log}\left( \boldsymbol{R}\right) $, $r\in 
\mathrm{N}$, and $0<\delta <\infty $. Then%
\begin{equation*}
\left\Vert \left( I-T_{\delta }\right) ^{r}f\right\Vert _{p\left( \cdot
\right) }\leq c_{10}^{r}2^{-r}\delta ^{r}\left\Vert f^{\left( r\right)
}\right\Vert _{p\left( \cdot \right) }\text{,}\quad f\in W_{L_{p\left( \cdot
\right) }}^{r}
\end{equation*}%
hold.
\end{lemma}

We will use notation $K_{r}\left( f,\delta ,p\left( \cdot \right) \right)
:=K_{r}\left( f,\delta ,L_{p\left( \cdot \right) }\right) _{L_{p\left( \cdot
\right) }}$ for $r\in \mathrm{N}$, $p\in P^{Log}\left( B\right) $, $\delta
>0 $ and $f\in L_{p\left( \cdot \right) }\left( B\right) $.

As a corollary of Transference Result we can obtain the following lemma.

\begin{lemma}
\label{bukun} Let $0<h\leq \delta <\infty $, $p\in P^{Log}\left( \boldsymbol{%
R}\right) $ and $f\in L_{p\left( \cdot \right) }$. Then%
\begin{equation}
\left\Vert \left( I-T_{h}\right) f\right\Vert _{p\left( \cdot \right) }\leq
c_{8}\left( 72,p^{+},c_{3}\left( p\right) \right) \left\Vert \left(
I-T_{\delta }\right) f\right\Vert _{p\left( \cdot \right) }  \label{bukunn}
\end{equation}%
holds.
\end{lemma}

In the following theorem we show that $K$-functional $K_{r}(f,\delta
,p\left( \cdot \right) )_{{p(\cdot )}}$ and $\Omega _{r}(f,\delta )_{p(\cdot
)}$ are equivalent.

\begin{theorem}
\label{teo1} Let $p(\cdot )\in P^{Log}\left( \boldsymbol{R}\right) $. If $%
L_{p\left( \cdot \right) }$, then the K-functional\emph{\ }$K_{r}\left(
f,\delta ,p\left( \cdot \right) \right) _{p\left( \cdot \right) }$\ and the
modulus $\Omega _{r}\left( f,\delta \right) _{p\left( \cdot \right) }$,\ are
equivalent, namely, 
\begin{eqnarray*}
\frac{1}{48c_{7}\left( c_{3}\left( p\right) \right) 2^{r}c_{5}\left(
p^{+},c_{3}\left( p\right) \right) } &\leq &\frac{K_{r}\left( f,\delta
,p\left( \cdot \right) \right) _{p\left( \cdot \right) }}{\Omega
_{r}(f,\delta )_{p(\cdot )}} \\
&\leq &48c_{7}\left( c_{3}\left( p\right) \right) \left\{
(2r)^{r}+2^{r}(34)^{r}\right\} c_{5}\left( p^{+},c_{3}\left( p\right)
\right) .
\end{eqnarray*}
\end{theorem}

\begin{theorem}
\label{rem1} For $p(\cdot )\in P^{Log}\left( \boldsymbol{R}\right) $, $%
f,g\in L_{p\left( \cdot \right) }$ and $\delta >0,$ the modulus of
smoothness $\Omega _{r}\left( f,\delta \right) _{p\left( \cdot \right) }$,
has the following properties:

\begin{enumerate}
\item $\Omega _{r}\left( f,\delta \right) _{p\left( \cdot \right) }$ is
non-negative; non-decreasing function of $\delta $;

\item For $f,g\in L_{p\left( \cdot \right) }$ and $\delta >0,$ 
\begin{equation}
\Omega _{r}(f+g,\delta )_{p(\cdot )}\leq \Omega _{r}(f,\delta )_{p(\cdot
)}+\Omega _{r}(g,\delta )_{p(\cdot )}.  \label{P2}
\end{equation}

\item For $f\in L_{p\left( \cdot \right) }$, 
\begin{equation}
\lim\limits_{\delta \rightarrow 0}\Omega _{r}(f,\delta )_{p(\cdot )}=0.
\label{P3}
\end{equation}
\end{enumerate}
\end{theorem}

As a corollary of Theorem \ref{teo1}

\begin{corollary}
\label{co} Let $p(\cdot )\in P^{Log}\left( \boldsymbol{R}\right) $. If $%
\delta ,\lambda \in \left( 0,1\right) $, $f\in L_{p\left( \cdot \right) }$,
then,%
\begin{equation*}
\frac{\Omega _{r}\left( f,\lambda \delta \right) _{p\left( \cdot \right) }}{%
\left( 1\text{+}\lfloor \lambda \rfloor \right) ^{r}\Omega _{r}\left(
f,\delta \right) _{p\left( \cdot \right) }}\leq \left( 48\right)
^{2}c_{7}^{2}\left( c_{3}\left( p\right) \right) 2^{r}c_{5}^{2}\left(
p^{+},c_{3}\left( p\right) \right) \left( (2r)^{r}\text{+}%
2^{r}(34)^{r}\right)
\end{equation*}%
holds.
\end{corollary}

\begin{theorem}
\label{jak} Let $p(\cdot )\in P^{Log}\left( \boldsymbol{R}\right) $, $r\in 
\mathrm{N}$, $\sigma >0$ and $f\in L_{p\left( \cdot \right) }$. Then,%
\begin{equation}
A_{\sigma }\left( f\right) _{p\left( \cdot \right) }\leq c_{11}\left\Vert
\left( I-T_{1/\sigma }\right) ^{r}f\right\Vert _{p\left( \cdot \right) }
\label{JJ}
\end{equation}%
with $c_{11}:=c_{11}(r,p^{+},c_{3}\left( p\right) ):=30\pi 8^{r}c_{5}\left(
p^{+},c_{3}\left( p\right) \right) c_{7}\left( c_{3}\left( p\right) \right)
c_{8}\left( r\right) $.
\end{theorem}

Now we present the inverse theorem.

\begin{theorem}
\label{Ters T}Let $p(\cdot )\in P^{Log}\left( \boldsymbol{R}\right) $, $r\in 
\mathrm{N}$, $\delta \in \left( 0,1\right) $ and $f\in L_{p\left( \cdot
\right) }$. Then,%
\begin{equation*}
\Omega _{r}\left( f,\delta \right) _{p(\cdot )}\leq c_{12}\delta ^{r}\left(
A_{0}\left( f\right) _{p\left( \cdot \right) }+\int\limits_{1/2}^{1/\delta
}u^{r-1}A_{u/2}\left( f\right) _{p\left( \cdot \right) }du\right) 
\end{equation*}%
holds with $c_{12}:=c_{12}\left( r,p^{+},c_{3}\left( p\right) \right)
:=c_{13}12c_{7}\left( c_{3}\left( p\right) \right) \left( 1+2^{2r-1}\right)
2^{r}$ where $c_{13}:=c_{13}\left( p^{+},c_{3}\left( p\right) \right)
:=2c_{5}\left( p^{+},c_{3}\left( p\right) \right) \left( 1+72c_{7}\left(
c_{3}\left( p\right) \right) c_{5}\left( p^{+},c_{3}\left( p\right) \right)
\right) .$
\end{theorem}

In this section we obtain Marchaud inequality.

\begin{theorem}
\label{Ters} Let $r,k\in \mathrm{N}$, $p\in P^{Log}\left( \boldsymbol{R}%
\right) $, $f\in L_{p\left( \cdot \right) }$ and $t\in \left( 0,1/2\right) $%
. Then,%
\begin{equation*}
\Omega _{r}\left( f,t\right) _{p(\cdot )}\leq c_{14}t^{r}\int_{t}^{1}\frac{%
\Omega _{r+k}\left( f,u\right) _{p(\cdot )}}{u^{r+1}}du
\end{equation*}%
holds with $c_{14}:=c_{14}(r,k,p^{+},c_{3}\left( p\right) ):=48c_{7}\left(
c_{3}\left( p\right) \right) C_{9}\left( r,k\right) c_{5}\left(
p^{+},c_{3}\left( p\right) \right) .$
\end{theorem}

\begin{theorem}
\label{crv} Let $p\in P^{Log}\left( \boldsymbol{R}\right) $, $r\in \mathrm{N}
$ and $f\in L_{p\left( \cdot \right) }$. If%
\begin{equation*}
\sum\limits_{\nu =0}^{\infty }\nu ^{k-1}A_{\nu /2}\left( f\right) _{p\left(
\cdot \right) }<\infty 
\end{equation*}%
holds for some $k\in \mathrm{N}$, then $f^{\left( k\right) }\in L_{p\left(
\cdot \right) }$ and%
\begin{equation}
\Omega _{r}\left( f^{\left( k\right) },\frac{1}{\sigma }\right) _{p\left(
\cdot \right) }\leq c_{14}\left( \frac{1}{\sigma ^{r}}\sum\limits_{\nu
=0}^{\lfloor \sigma \rfloor }\left( \nu +1\right) ^{r+k-1}A_{\nu /2}\left(
f\right) _{p\left( \cdot \right) }+\sum\limits_{\nu =\lfloor \sigma \rfloor
+1}^{\infty }\nu ^{k-1}A_{\nu /2}\left( f\right) _{p\left( \cdot \right)
}\right)   \label{Sinv}
\end{equation}%
with $c_{14}:=c_{14}(r,k,p^{+},c_{3}\left( p\right) ):=48c_{7}\left(
c_{3}\left( p\right) \right) c_{5}\left( p^{+},c_{3}\left( p\right) \right)
2^{2k+r+2}$.
\end{theorem}

\subsection{Proofs of the results of section 3}

\begin{proof}[Proof of Lemma \protect\ref{lemma1}]
We note that (see \cite{AG}) the following inequality 
\begin{equation}
\left\Vert \left( I-T_{\delta }\right) f\right\Vert _{p\left( \cdot \right)
}\leq 2^{-1}c_{10}\delta \left\Vert f^{\text{ }\prime }\right\Vert _{p\left(
\cdot \right) }\text{,}\quad \delta >0  \label{eqn6*}
\end{equation}%
holds for $f\in L_{p\left( \cdot \right) }$. Then%
\begin{equation*}
\Omega _{r}\left( f,\delta \right) _{p\left( \cdot \right) }=\left\Vert
\left( I-T_{\delta }\right) ^{r}f\right\Vert _{p\left( \cdot \right) }\leq
...\leq 2^{-r}c_{10}^{r}\delta ^{r}\left\Vert f^{(r)}\right\Vert _{p\left(
\cdot \right) }\text{, }\delta >0
\end{equation*}%
for $f\in W_{L_{p\left( \cdot \right) }}^{r}$.
\end{proof}

\begin{proof}[Proof of Theorem \protect\ref{teo1}]
For any $g\in W_{L_{p\left( \cdot \right) }}^{r}\left( \Omega \right) $ we
have $F_{g}\in C^{r}\left( \Omega \right) $. Since $F_{f}$ is linear in $f,$%
\begin{equation*}
\left( I-T_{\delta }\right) ^{r}F_{f}=F_{\left( I-T_{\delta }\right) ^{r}f}%
\text{ \ \ and }\left( F_{g}\right) ^{\left( r\right) }=F_{g^{\left(
r\right) }}
\end{equation*}%
using Theorem \ref{tra} we obtain%
\begin{eqnarray*}
\left\Vert \left( I-T_{\delta }\right) ^{r}f\right\Vert _{p\left( \cdot
\right) } &\leq &24c_{7}\left( c_{3}\left( p\right) \right) \left\Vert
F_{\left( I-T_{\delta }\right) ^{r}f}\right\Vert _{C\left( \Omega \right) }
\\
&=&24c_{7}\left( c_{3}\left( p\right) \right) \left\Vert \left( I-T_{\delta
}\right) ^{r}F_{f}\right\Vert _{C\left( \Omega \right) } \\
&\leq &24c_{7}\left( c_{3}\left( p\right) \right) 2^{r}K_{r}\left(
F_{f},\delta ,C\left( \Omega \right) \right) _{C\left( \Omega \right) } \\
&\leq &24c_{7}\left( c_{3}\left( p\right) \right) 2^{r}\left\{ \left\Vert
F_{f}-F_{g}\right\Vert _{C\left( \Omega \right) }+\delta ^{r}\left\Vert
\left( F_{g}\right) ^{\left( r\right) }\right\Vert _{C\left( \Omega \right)
}\right\} \\
&=&24c_{7}\left( c_{3}\left( p\right) \right) 2^{r}\left\{ \left\Vert
F_{(f-g)}\right\Vert _{C\left( \Omega \right) }+\delta ^{r}\left\Vert
F_{g^{\left( r\right) }}\right\Vert _{C\left( \Omega \right) }\right\} \\
&\leq &48c_{7}\left( c_{3}\left( p\right) \right) 2^{r}c_{5}\left(
p^{+},c_{3}\left( p\right) \right) \left\{ \left\Vert f-g\right\Vert
_{p\left( \cdot \right) }+\delta ^{r}\left\Vert g^{\left( r\right)
}\right\Vert _{p\left( \cdot \right) }\right\} .
\end{eqnarray*}%
Taking infimum and considering definition of \textit{K}-functional one gets%
\begin{equation*}
\left\Vert \left( I-T_{\delta }\right) ^{r}f\right\Vert _{p\left( \cdot
\right) }\leq 48c_{7}\left( c_{3}\left( p\right) \right) 2^{r}c_{5}\left(
p^{+},c_{3}\left( p\right) \right) K_{r}\left( f,\delta ,p\left( \cdot
\right) \right) _{p\left( \cdot \right) }.
\end{equation*}%
Now we consider the opposite direction of the last inequality. For%
\begin{equation*}
g\left( \cdot \right) =\sum\limits_{l=1}^{r}\left( -1\right) ^{l-1}\binom{r}{%
l}T_{\delta }^{2rl}f\left( \cdot \right)
\end{equation*}%
we have%
\begin{eqnarray*}
K_{r}\left( f,\delta ,p\left( \cdot \right) \right) _{p\left( \cdot \right)
} &\leq &\left\Vert f-g\right\Vert _{p\left( \cdot \right) }+\delta
^{r}\left\Vert \frac{d^{r}}{dx^{r}}g\right\Vert _{p\left( \cdot \right) } \\
&\leq &24c_{7}\left( c_{3}\left( p\right) \right) \left\{ \left\Vert
F_{(f-g)}\right\Vert _{C\left( \Omega \right) }+\delta ^{r}\left\Vert
F_{g^{\left( r\right) }}\right\Vert _{C\left( \Omega \right) }\right\} \\
&=&24c_{7}\left( c_{3}\left( p\right) \right) \left\{ \left\Vert
F_{f}-F_{g}\right\Vert _{C\left( \Omega \right) }+\delta ^{r}\left\Vert
\left( F_{g}\right) ^{\left( r\right) }\right\Vert _{C\left( \Omega \right)
}\right\}
\end{eqnarray*}%
\begin{equation*}
\leq 24c_{7}\left( c_{3}\left( p\right) \right) \left\{ \left\Vert \left( I%
\text{-}T_{\delta }^{2r}\right) ^{r}F_{f}\right\Vert _{C\left( \Omega
\right) }\text{+}\delta ^{r}\left\Vert \left( \sum\limits_{l=1}^{r}\left( 
\text{-}1\right) ^{l\text{-}1}\binom{r}{l}T_{\delta }^{2rl}F_{f}\right)
^{\left( r\right) }\right\Vert _{C\left( \Omega \right) }\right\}
\end{equation*}%
\begin{equation*}
=24c_{7}\left( c_{3}\left( p\right) \right) \left\{ \left\Vert \left(
I-T_{\delta }^{2r}\right) ^{r}F_{f}\right\Vert _{C\left( \Omega \right)
}+\sum\limits_{l=1}^{r}\left\vert \binom{r}{l}\right\vert \delta
^{r}\left\Vert \left( T_{\delta }^{2rl}F_{f}\right) ^{\left( r\right)
}\right\Vert _{C\left( \Omega \right) }\right\}
\end{equation*}%
\begin{equation*}
\leq 24c_{7}\left( c_{3}\left( p\right) \right) \left\{ (2r)^{r}\left\Vert
\left( I-T_{\delta }\right) ^{r}F_{f}\right\Vert _{C\left( \Omega \right)
}+2^{r}(34)^{r}\left\Vert \left( I-T_{\delta }\right) ^{r}F_{f}\right\Vert
_{C\left( \Omega \right) }\right\}
\end{equation*}%
\begin{equation*}
=24c_{7}\left( c_{3}\left( p\right) \right) \left\{
(2r)^{r}+2^{r}(34)^{r}\right\} \left\Vert F_{\left( I-T_{\delta }\right)
^{r}f}\right\Vert _{C\left( \Omega \right) }
\end{equation*}%
\begin{equation*}
\leq 48c_{7}\left( c_{3}\left( p\right) \right) \left\{
(2r)^{r}+2^{r}(34)^{r}\right\} c_{5}\left( p^{+},c_{3}\left( p\right)
\right) \left\Vert \left( I-T_{\delta }\right) ^{r}f\right\Vert _{p\left(
\cdot \right) }.
\end{equation*}
\end{proof}

\begin{proof}[Proof of Theorem \protect\ref{rem1}]
Properties (1) and (2), by definition of $\Omega _{r}\left( f,\delta \right)
_{p\left( \cdot \right) }$ and the triangle inequality of $L_{p\left( \cdot
\right) }$ are clearly valid. By using \cite[Theorem 10.1]{Dit} and \cite[%
Lemma 2]{Israfil}, the relation \eqref{P3} is satisfied.
\end{proof}

\begin{proof}[Proof of Corollary \protect\ref{co}]
We have%
\begin{equation*}
\frac{\Omega _{r}\left( f,\lambda \delta \right) _{p\left( \cdot \right) }}{%
\left( 1+\lfloor \lambda \rfloor \right) ^{r}\Omega _{r}\left( f,\delta
\right) _{p\left( \cdot \right) }}\leq \frac{48c_{7}\left( c_{3}\left(
p\right) \right) 2^{r}c_{5}\left( p^{+},c_{3}\left( p\right) \right) }{%
\left( 1+\lfloor \lambda \rfloor \right) ^{r}}\frac{K_{r}\left( f,\lambda
\delta ,p\left( \cdot \right) \right) _{p\left( \cdot \right) }}{\Omega
_{r}\left( f,\delta \right) _{p\left( \cdot \right) }}
\end{equation*}%
\begin{eqnarray*}
&\leq &\frac{\left( 48\right) ^{2}c_{7}^{2}\left( c_{3}\left( p\right)
\right) 2^{r}c_{5}^{2}\left( p^{+},c_{3}\left( p\right) \right) }{\left(
1+\lfloor \lambda \rfloor \right) ^{r}}\frac{\left( 1+\lfloor \lambda
\rfloor \right) ^{r}}{1}\left\{ (2r)^{r}+2^{r}(34)^{r}\right\} \\
&=&\left( 48\right) ^{2}c_{7}^{2}\left( c_{3}\left( p\right) \right)
2^{r}c_{5}^{2}\left( p^{+},c_{3}\left( p\right) \right) \left\{
(2r)^{r}+2^{r}(34)^{r}\right\} .
\end{eqnarray*}
\end{proof}

\begin{proof}[Proof of Theorem \protect\ref{jak}]
First we obtain%
\begin{equation}
A_{2\sigma }\left( f\right) _{p\left( \cdot \right) }\leq 30\pi
8^{r}c_{5}\left( p^{+},c_{3}\left( p\right) \right) c_{7}\left( c_{3}\left(
p\right) \right) c_{8}\left( r\right) \left\Vert \left( I\text{-}T_{1/\left(
2\sigma \right) }\right) ^{r}f\right\Vert _{p\left( \cdot \right) }
\label{JE}
\end{equation}%
and (\ref{JJ}) follows from (\ref{JE}). Let $g_{\sigma }$ be an exponential
type entire function of degree $\leq \sigma $, belonging to $\mathcal{C}(%
\boldsymbol{R})$, as best approximation of $F_{f}\in \mathcal{C}(\boldsymbol{%
R})$. Since $F_{V_{\sigma }f}=V_{\sigma }F_{f}$ and $V_{\sigma }g_{\sigma
}=g_{\sigma }$, there holds%
\begin{equation*}
A_{2\sigma }\left( f\right) _{p\left( \cdot \right) }\leq \left\Vert
f-V_{\sigma }f\right\Vert _{p\left( \cdot \right) }\leq 24c_{7}\left(
c_{3}\left( p\right) \right) \left\Vert F_{f-V_{\sigma }f}\right\Vert _{%
\mathcal{C}(\boldsymbol{R})}
\end{equation*}%
\begin{equation*}
=24c_{7}\left( c_{3}\left( p\right) \right) \left\Vert F_{f}-V_{\sigma
}F_{f}\right\Vert _{\mathcal{C}(\boldsymbol{R})}
\end{equation*}%
\begin{equation*}
=24c_{7}\left( c_{3}\left( p\right) \right) \left\Vert F_{f}-g_{\sigma
}+g_{\sigma }-V_{\sigma }F_{f}\right\Vert _{\mathcal{C}(\boldsymbol{R})}
\end{equation*}%
\begin{equation*}
=24c_{7}\left( c_{3}\left( p\right) \right) \left\Vert F_{f}-g_{\sigma
}+V_{\sigma }g_{\sigma }-V_{\sigma }F_{f}\right\Vert _{\mathcal{C}(%
\boldsymbol{R})}
\end{equation*}%
\begin{equation*}
\leq 24c_{7}\left( c_{3}\left( p\right) \right) (A_{\sigma }\left(
F_{f}\right) _{\mathcal{C}(\boldsymbol{R})}+\frac{3}{2}A_{\sigma }\left(
F_{f}\right) _{\mathcal{C}(\boldsymbol{R})})=12c_{7}\left( c_{3}\left(
p\right) \right) A_{\sigma }\left( F_{f}\right) _{\mathcal{C}(\boldsymbol{R}%
)}.
\end{equation*}

For any $g\in W_{\mathcal{C}(\boldsymbol{R})}^{r}$%
\begin{equation*}
A_{\sigma }\left( u\right) _{\mathcal{C}(\boldsymbol{R})}\leq A_{\sigma
}\left( u-g\right) _{\mathcal{C}(\boldsymbol{R})}+A_{\sigma }\left( g\right)
_{\mathcal{C}(\boldsymbol{R})}
\end{equation*}%
\begin{equation*}
\leq \left\Vert u-g\right\Vert _{\mathcal{C}(\boldsymbol{R})}+\frac{5\pi }{4}%
\frac{4^{r}}{\sigma ^{r}}\left\Vert \frac{d^{r}}{dx^{r}}g\right\Vert _{%
\mathcal{C}(\boldsymbol{R})}
\end{equation*}%
\begin{equation*}
\leq \frac{5\pi 4^{r}}{4}K_{r}\left( u,\frac{1}{\sigma },\mathcal{C}(%
\boldsymbol{R})\right) _{\mathcal{C}(\boldsymbol{R})}\leq \frac{5\pi 8^{r}}{4%
}K_{r}\left( u,\frac{1}{2\sigma },\mathcal{C}(\boldsymbol{R})\right) _{%
\mathcal{C}(\boldsymbol{R})}
\end{equation*}%
\begin{equation*}
\leq \frac{5\pi 8^{r}}{4}c_{8}\left( r\right) \left\Vert \left( I-T_{\frac{1%
}{2\sigma }}\right) ^{r}u\right\Vert _{\mathcal{C}(\boldsymbol{R})}.
\end{equation*}%
Therefore%
\begin{eqnarray*}
A_{2\sigma }\left( f\right) _{p\left( \cdot \right) } &\leq &12c_{7}\left(
c_{3}\left( p\right) \right) A_{\sigma }\left( F_{f}\right) _{\mathcal{C}(%
\boldsymbol{R})} \\
&\leq &15\pi 8^{r}c_{7}\left( c_{3}\left( p\right) \right) c_{8}\left(
r\right) \left\Vert \left( I-T_{\frac{1}{2\sigma }}\right)
^{r}F_{f}\right\Vert _{\mathcal{C}(\boldsymbol{R})}
\end{eqnarray*}%
\begin{equation*}
=15\pi 8^{r}c_{7}\left( c_{3}\left( p\right) \right) c_{8}\left( r\right)
\left\Vert F_{\left( I-T_{1/\left( 2\sigma \right) }\right)
^{r}f}\right\Vert _{\mathcal{C}(\boldsymbol{R})}
\end{equation*}%
\begin{equation*}
\leq 30\pi 8^{r}c_{5}\left( p^{+},c_{3}\left( p\right) \right) c_{7}\left(
c_{3}\left( p\right) \right) c_{8}\left( r\right) \left\Vert \left(
I-T_{1/\left( 2\sigma \right) }\right) ^{r}f\right\Vert _{p\left( \cdot
\right) }.
\end{equation*}
\end{proof}

\begin{proof}[Proof of Theorem \protect\ref{Ters T}]
Let $g_{\sigma }$ be an exponential type entire function of degree $\leq
\sigma $, belonging to $L^{p(\cdot )}$, as best approximation of $f\in
L^{p(\cdot )}$. Then%
\begin{eqnarray*}
\Omega _{r}\left( f,\delta \right) _{p(\cdot )} &=&\left\Vert \left(
I-T_{\delta }\right) ^{r}f\right\Vert _{p\left( \cdot \right) }\leq
24c_{7}\left( c_{3}\left( p\right) \right) \left\Vert F_{\left( I-T_{\delta
}\right) ^{r}f}\right\Vert _{\mathcal{C}(\boldsymbol{R})} \\
&=&24c_{7}\left( c_{3}\left( p\right) \right) \left\Vert \left( I-T_{\delta
}\right) ^{r}F_{f}\right\Vert _{\mathcal{C}(\boldsymbol{R})}
\end{eqnarray*}%
\begin{equation*}
\leq 12c_{7}\left( c_{3}\left( p\right) \right) \left( 1+2^{2r-1}\right)
2^{r}\delta ^{r}\left( A_{0}\left( F_{f}\right) _{\mathcal{C}(\boldsymbol{R}%
)}+\int_{1/2}^{1/\delta }u^{r-1}A_{u}\left( F_{f}\right) _{\mathcal{C}(%
\boldsymbol{R})}du\right) 
\end{equation*}%
\begin{equation*}
\leq c_{13}12c_{7}\left( c_{3}\left( p\right) \right) \left(
1+2^{2r-1}\right) 2^{r}\delta ^{r}\left( A_{0}\left( f\right) _{p\left(
\cdot \right) }+\int_{1/2}^{1/\delta }u^{r-1}A_{u/2}\left( f\right)
_{p\left( \cdot \right) }du\right) 
\end{equation*}%
because%
\begin{equation*}
A_{2\sigma }\left( F_{f}\right) _{\mathcal{C}(\boldsymbol{R})}\leq
\left\Vert F_{f}\text{-}V_{\sigma }F_{f}\right\Vert _{\mathcal{C}(%
\boldsymbol{R})}=\left\Vert F_{f-V_{\sigma }f}\right\Vert _{\mathcal{C}(%
\boldsymbol{R})}\leq 2c_{5}\left( p^{+},c_{3}\left( p\right) \right)
\left\Vert f\text{-}V_{\sigma }f\right\Vert _{p\left( \cdot \right) }
\end{equation*}%
\begin{eqnarray*}
&=&2c_{5}\left( p^{+},c_{3}\left( p\right) \right) \left\Vert f-g_{\sigma
}+g_{\sigma }-V_{\sigma }f\right\Vert _{p\left( \cdot \right) } \\
&\leq &2c_{5}\left( p^{+},c_{3}\left( p\right) \right) \left( \left\Vert
f-g_{\sigma }\right\Vert _{p\left( \cdot \right) }+\left\Vert V_{\sigma
}g_{\sigma }-V_{\sigma }f\right\Vert _{p\left( \cdot \right) }\right)  \\
&\leq &2c_{5}\left( p^{+},c_{3}\left( p\right) \right) \left( \left\Vert
f-g_{\sigma }\right\Vert _{p\left( \cdot \right) }+72c_{7}\left( c_{3}\left(
p\right) \right) c_{5}\left( p^{+},c_{3}\left( p\right) \right) \left\Vert
g_{\sigma }-f\right\Vert _{p\left( \cdot \right) }\right)  \\
&=&2c_{5}\left( p^{+},c_{3}\left( p\right) \right) \left( 1+72c_{7}\left(
c_{3}\left( p\right) \right) c_{5}\left( p^{+},c_{3}\left( p\right) \right)
\right) A_{\sigma }\left( f\right) _{_{p\left( \cdot \right) }}.
\end{eqnarray*}
\end{proof}

\begin{proof}[Proof of Theorem \protect\ref{Ters}]
Let $g_{\sigma }$ be an exponential type entire function of degree $\leq
\sigma $, belonging to $L^{p(\cdot )}$, as best approximation of $f\in
L_{p\left( \cdot \right) }$. Then%
\begin{eqnarray*}
\Omega _{r}\left( f,t\right) _{p(\cdot )} &=&\left\Vert \left(
I-T_{t}\right) ^{r}f\right\Vert _{p\left( \cdot \right) }\leq 24c_{7}\left(
c_{3}\left( p\right) \right) \left\Vert F_{\left( I-T_{t}\right)
^{r}f}\right\Vert _{\mathcal{C}(\boldsymbol{R})} \\
&=&24c_{7}\left( c_{3}\left( p\right) \right) \left\Vert \left(
I-T_{t}\right) ^{r}F_{f}\right\Vert _{\mathcal{C}(\boldsymbol{R})}
\end{eqnarray*}%
\begin{equation*}
\leq 24c_{7}\left( c_{3}\left( p\right) \right) C_{9}\left( r,k\right)
t^{r}\int_{t}^{1}\frac{\left\Vert \left( I-T_{t}\right)
^{r+k}F_{f}\right\Vert _{\mathcal{C}(\boldsymbol{R})}}{u^{r+1}}du
\end{equation*}%
\begin{equation*}
=24c_{7}\left( c_{3}\left( p\right) \right) C_{9}\left( r,k\right)
t^{r}\int_{t}^{1}\frac{\left\Vert F_{\left( I-T_{t}\right)
^{r+k}f}\right\Vert _{\mathcal{C}(\boldsymbol{R})}}{u^{r+1}}du
\end{equation*}%
\begin{equation*}
\leq 48c_{7}\left( c_{3}\left( p\right) \right) C_{9}\left( r,k\right)
c_{5}\left( p^{+},c_{3}\left( p\right) \right) t^{r}\int_{t}^{1}\frac{%
\left\Vert \left( I-T_{t}\right) ^{r+k}f\right\Vert _{p(\cdot )}}{u^{r+1}}du
\end{equation*}%
\begin{equation*}
=48c_{7}\left( c_{3}\left( p\right) \right) C_{9}\left( r,k\right)
c_{5}\left( p^{+},c_{3}\left( p\right) \right) t^{r}\int_{t}^{1}\frac{\Omega
_{r+k}\left( f,t\right) _{p(\cdot )}}{u^{r+1}}du.
\end{equation*}
\end{proof}

\begin{proof}[Proof of Theorem \protect\ref{crv}]
Proof of (\ref{Sinv}) is similar to that of proof of Theorem \ref{Ters}.
\end{proof}


\begin{thebibliography}{99}
\bibitem{acis} F. Abdullaev, S. Chaichenko, M. Imashgizi, and A. Shidlich,
Direct and inverse approximation theorems in the weighted Orlicz-type spaces
with a variable exponent, Turk. J. Math., 44 (2020), 284-299.

\bibitem{ascs} F. Abdullaev, A. Shidlich and S. Chaichenko, Direct and
inverse approximation theorems of functions in the Orlicz type spaces, Math.
Slovaca, 69 (2019), 1367-1380.

\bibitem{aosss} F. Abdullaev, N. \"{O}zkaratepe, V. Savchuk and A. Shidlich,
Exact constants in direct and inverse approximation theorems for functions
of several variables in the spaces $S_{p}$, F\.{I}LOMAT, 33 (2019),
1471-1484.

\bibitem{Ak} N. I. Ackhiezer, Theory of approximation, Fizmatlit, Moscow,
1965; English transl. of 2nd ed. Frederick Ungar, New York, 1956.

\bibitem{ra11u} R. Akg\"{u}n, Approximation of Functions of Weighted
Lebesgue and Smirnov Spaces, Mathematica (Cluj), Tome 54 (77), No: Special
(2012), pp. 25-36.

\bibitem{AK1} R. Akg\"{u}n Sharp Jackson and converse theorems of
trigonometric approximation in weighted Lebesgue spaces, Proc. A. Razmadze
Math. Inst., 152 (2010), pp. 1-18.

\bibitem{AK2} R. Akg\"{u}n, Inequalities for one sided approximation in
Orlicz spaces, Hacet. J. Math. Stat., Volume: 40 Issue: 2 (2011), pp.
231-240.

\bibitem{AK3} R. Akg\"{u}n, Some convolution inequalities in Musielak Orlicz
spaces, Proc. Inst. Math. Mech., NAS Azerbaijan, 42 (2016), No: 2, 279-291.

\bibitem{AkgArx17} R. Akg\"{u}n, Approximation properties of Bernstein's
singular integrals in variable exponent Lebesgue spaces on the real axis,
Communications Faculty of Sciences University of Ankara Series A1
Mathematics and Statistics, In press, arXiv:1210.4714v3 [math.CA],
https://doi.org/10.48550/arXiv.1210.4714

\bibitem{AG} R. Akg\"{u}n; A. Ghorbanalizadeh, "Approximation by integral
functions of finite degree in variable exponent Lebesgue spaces on the real
axis." Turkish Journal of Mathematics 42, (2018), no. 4, 1887--1903.

\bibitem{ahak} A.H. Av\c{s}ar and H. Ko\c{c}, Jackson and Stechkin type
inequalities of trigonometric approximation in $A\_p,q(.)\symbol{94}w$,$%
\theta $, Turk J Math (2018) 42: 2979-2993.

\bibitem{ayey} A.H. Av\c{s}ar and Y.E. Yildirir, On the trigonometric
approximation of functions in weighted Lorentz spaces using Cesaro
submethod, Novi Sad J. Math. Vol. 48, No. 2, 2018, 41-54.

\bibitem{bbsv06} C. Bardaro, P.L. Butzer, R.L. Stens and G. Vinti,
Approximation error of the Whittaker cardinal series in terms of an averaged
modulus of smoothness covering discontinuous signals, J. Math. Anal. Appl.
316 (2006) 269-306.

\bibitem{B46} S.N. Bernstein, Sur la meilleure approximation sur tout l'axe
reel des fonctions continues par des fonctions entieres de degre n. I, C.R.
(Doklady) Acad. Sci. URSS (N.S.) 51 (1946), 331-334.

\bibitem{Ber} S. N. Bernstein, \textit{\ }Collected works, M. Vol. I, Izdat.
Akad. Nauk SSSR, Moscow, 1952., 11-104.

\bibitem{UF13} D. Cruz-Uribe, A. Fiorenza, Variable Lebesgue Spaces,
Foundations and Harmonic Analysis, Birkhauser, Applied and Numerical
Harmonic Analysis, 2013.

\bibitem{cuf} D. Cruz-Uribe, A. Fiorenza, Approximate identities in variable
Lp spaces, Mathematische Nachrichten, 280 (2007), No:3, 256-270.

\bibitem{devore} R. A. Devore, G. G. Lorentz, Constructive Approximation,
Springer-Verlag, (1993).

\bibitem{DHHR11} L. Diening, P. Harjulehto, P. Hasto and M. Ru\v{z}i\v{c}ka,
Lebesgue and Sobolev spaces with variable exponents, Lecture Notes in Math.,
vol. 2017, Springer, Berlin, Heidelberg, 2011.

\bibitem{ld-mr02} L. Diening and M. Ru\v{z}i\v{c}ka, Calderon--Zymund
operators on generalized Lebesgue spaces $L^{p(x)}$ and problems related to
fluid dynamics, preprint, Mathematische Fak\"{u}ltat,
Albert-Ludwings-Universit\"{a}t Freiburg, 21/2002, 04.07.2002, 1-20, 2002.

\bibitem{z76} Z. Ditzian, Inverse theorems for functions in $L^{p}$ and
other spaces, Proc. Amer. Mafh. Soc. 54 (1976), 80-82.

\bibitem{Dit} Z. Ditzian and K. G. Ivanov, Strong converse inequalities,
Journal D'analyse mathematique \textbf{61}:1, 1993, 61-111.

\bibitem{day} A. Dogu, A.H. Avsar and Y.E. Yildirir, Some inequalities about
convolution and trigonometric approximation in weighted Orlicz spaces,
Proceedings of the Institute of Mathematics and Mechanics, National Academy
of Sciences of Azerbaijan, Volume 44, Number 1, 2018, Pages 107-115.

\bibitem{drh20} D. Drihem, Restricted boundedness of translation operators
on variable Lebesgue spaces, https://doi.org/10.48550/arXiv.1507.08089

\bibitem{DQR03} D. P. Dryanov, M. A. Qazi, and Q. I. Rahman, Entire
functions of exponential type in Approximation Theory, In:\ Constructive
Theory of Functions, Varna 2002 (B. Bojanov, Ed.), DARBA, Sofia, 2003, pp.
86-135.

\bibitem{xf-dz01} X. Fan and D. Zhao, On the spaces $L^{p(x)}(\Omega )$ and $%
W^{m,p(x)}(\Omega )$, J. Math. Anal. Appl. \textbf{263}:2 (2001), 424--446.

\bibitem{GI2} A. Guven and D.M. Israfilov, Trigonometric approximation in
generalized Lebesgue spaces\textit{\ }$L^{p(x)}$\textit{,} J. Math. Inequal.
4 (2010), no. 2, 285--299.

\bibitem{HH} P. Harjulehto and P. H\"{a}st\"{o}, Orlicz Spaces and
Generalized Orlicz Spaces, Lecture Notes in Mathematics, vol. 2236,
Springer, 2019, X+168.

\bibitem{hh76} H. Hudzik, On generalized Orlicz--Sobolev space, Funct.
Approximatio Comment. Math. \textbf{4} (1976), 37--51.

\bibitem{II3} I.I. Ibragimov, Teoriya priblizheniya tselymi funktsiyami%
\textit{.} (Russian) [The theory of approximation by entire functions]
"Elm\textquotedblright , Baku, 1979. 468 pp.

\bibitem{Ja} S.Z. Jafarov, Linear Methods for Summing Fourier Series and
Approximation in Weighted Lebesgue Spaces with Variable Exponents. Ukrainian
Mathematical Journal. 2015; 66(10): 1509---1518.

\bibitem{Ja1} S.Z. Jafarov, Approximation by trigonometric polynomials in
subspace of variable exponent grand Lebesgue spaces. Global Journal of
Mathematics, 2016; 8(2): 836--843.

\bibitem{ja2} S. Z. Jafarov, Ul'yanov type inequalities for moduli of
smoothness, Applied Mathematics E-Notes, 12 (2012), 221-227.

\bibitem{ja3} S. Z. Jafarov, S. M. Nikolskii type inequality and estimation
between the best approximations of a function in norms of di\textcurrency %
erent spaces. Math. Balkanica (N.S.) 21 (2007), no. 1-2, 173-182

\bibitem{AI} D. M. Israfilov and R. Akg\"{u}n, Approximation by polynomials
and rational functions in weighted rearrangement invariant spaces, J. Math.
Anal. Appl. 346 (2008), 489-500.

\bibitem{Israfil} D. M. Israfilov and A. Testici, Approximation problems in
the Lebesgue spaces with variable exponent, Journal of Mathematical Analysis
and Applications \textit{459}:1, 2018, 112--123.

\bibitem{it16} D. M. Israfilov and A. Testici, Approximation by
Faber--Laurent rational functions in Lebesgue spaces with variable exponent,
Indag. Mat. 27 (2016), No:4, 914-922.

\bibitem{iy} D. M. Israfilov and E. Yirtici, Convolutions and best
approximations in variable exponent Lebesgue spaces, Math. Reports 18(68)
(2016), No:4 , 497-508.

\bibitem{k} H. Koc, Simultaneous approximation by polynomials in Orlicz
spaces generated by quasiconvex Young functions, Kuwait J. Sci. 43 (2016);
No:4, 18-31.

\bibitem{kosa03gmj} V. Kokilashvili and S. Samko, Singular integrals in
weighted Lebesgue spaces with variable exponent, Georgian Math. J., 10
(2003), N0: 1, 145-156.

\bibitem{zok-jr91} Z. O. Kov\'{a}\v{c}ik and J. R\'{a}kosnik, On spaces $%
L^{p(x)}$ and $W^{k,p(x)}$, Czechoslovak Math. J. \textbf{41(116)}:4 (1991),
592--618.

\bibitem{NF} F.G. Nasibov, Approximation in $L_{2}$ by entire
functions.(Russian) Akad. Nauk Azerbaidzhan. SSR Dokl. 1986; 42(4): 3---6.

\bibitem{Ni} S. M. Nikolskii, Inequalities for entire functions of finite
degree and their application to the theory of differentiable functions of
several variables, Amer. Math. Soc. Transl. Ser. 2, 80 (1969), 1-38, (Trudy
Mat. Inst. Steklov 38 (1951), 211-278).

\bibitem{LiDo} A.A. Ligun and V.G. Doronin, Exact constants in Jackson-type
inequalities for the $L_{2}$-approximation on a straight line. (Russian) Ukra%
\"{\i}n. Mat. Zh. 2009; 61(1): 92--98; translation in Ukrainian Math. J.
2009; 61(1): 112--120.

\bibitem{worl31} W. Orlicz, \"{U}ber konjugierte Exponentenfolgen, Studia
Math. \textbf{3} (1931), 200-212.

\bibitem{RN} R. Paley and N. Wiener,\textit{\ }Fourier Transforms in the
Complex Domain, Amer. Math. Soc., 1934.

\bibitem{Po} V. Yu. Popov, Best mean square approximations by entire
functions of exponential type. (Russian) Izv. Vys\v{s}. Ucebn. Zaved.
Matematika. 1972; 121(6): 65---73.

\bibitem{krr-mr-96} K. R. Rajagopal and M. Ru\v{z}i\v{c}ka, On the modeling
elektroreological materials, Mech. Res. Commun. \textbf{23}:4 (1996),
401--407.

\bibitem{mr00} M. Ru\v{z}i\v{c}ka, Electrorheological Fluids: Modeling and
Mathematical Theory, Lecture Notes in Mathematics, 1748. Springer-Verlag,
Berlin, 2000.

\bibitem{sgs94} S. Samko, Differentiation and integration of variable order
and the spaces $L^{p(x)}$, in: Operator theory for complex and hypercomplex
analysis (Mexico City, 1994), 203--219, Contemp. Math., 212, Amer. Math.
Soc., Providence, RI, 1998.

\bibitem{iis79} I. I. Sharapudinov, The topology of the space $%
L^{p(t)}([0,\,1])$, (Russian) Mat. Zametki \textbf{26}:4 (1979), 613--632.

\bibitem{Sh12} I. I. Sharapudinov, Some questions in the theory of
approximation in Lebesgue spaces with variable exponent, Itogi Nauki. Yug
Rossii. Mat. Monografiya, vol. 5, Southern Institute of Mathematics of the
Vladikavkaz Sceince Centre of the Russian Academy of Sciences and the
Government of the Republic of North Ossetia-Alania, Vladikavkaz 2012, 267
pp. Russian.

\bibitem{AFT} A.F. Timan, Theory of approximation of functions of a real
variable. Translated from the Russian by J. Berry. English translation
edited and editorial preface by J. Cossar. International Series of
Monographs in Pure and Applied Mathematics, Vol. 34, The Macmillan Co., New
York: A Pergamon Press Book. 1963.

\bibitem{MFT} M.F.Timan, The approximation of functions defined on the whole
real axis by entire functions of exponential type. Izv. Vyssh. Uchebn.
Zaved. Mat. 1968; 2: 89---101.

\bibitem{T81} R. Taberski, Approximation by entire functions of exponential
type, 1981, Demonstr. Math. 14, 151-181 (1981).

\bibitem{T86} R. Taberski, Contributions to fractional calculus and
exponential approximation, 1986, Funct. Approximatio, Comment. Math. 15,
81-106 (1986).

\bibitem{ssv} S. S. Volosivets, Approximation of functions and their
conjugates in variable Lebesgue spaces, Sbornik: Mathematics, 208 (2017),
No:1, 44-59.

\bibitem{yeh} J. Yeh, Real analysis: theory of measure and integration, 2nd
ed., 2006.

\bibitem{vz86} V. V. Zhikov, Averaging of functionals of the calculus of
variations and elasticity theory, Izv. Akad. Nauk SSSR Ser. Mat. \textbf{50}%
:4 (1986), 675--710 (in Russian).

\bibitem{zhk} V. V. Zhuk, Approximation of periodic functions, LGU Press,
1982.
\end{thebibliography}
\end{document}